\documentclass[11pt,a4paper,twoside]{article}
\usepackage{titling}
\usepackage{titletoc}
\usepackage{url}

\usepackage[utf8]{inputenc}
\usepackage[T1]{fontenc}

\usepackage[french]{babel}

\usepackage{amsfonts}
\usepackage{amssymb}
\usepackage{amsthm}
\usepackage{mathrsfs}
\usepackage{stmaryrd}
\usepackage[all]{xy}
\usepackage{makeidx}

\usepackage{xspace}
\usepackage{etoolbox}

\patchcmd{\tableofcontents}{\MakeUppercase}{}{}{}
\patchcmd{\tableofcontents}{\MakeUppercase}{}{}{}

\makeatletter\def\th@plain{\slshape}\makeatother
\makeatletter\patchcmd{\th@remark}{\itshape}{\slshape}{}{}\makeatother

\newcounter{bidon}

\newcommand \sibrouillon[1]{}

\usepackage[bookmarksopen=false,breaklinks=true,%
      backref=page,pagebackref=true,plainpages=false,%
      hyperindex=true,pdfstartview=FitH,colorlinks=true,%
      pdfpagelabels=true,linkcolor=blue,%
      citecolor=red,urlcolor=red,
      ]%
   {hyperref}

\begin{document}

\thispagestyle{empty}
~ 
\vspace{.5cm}

\begin{center} 
   {\Large \bf The constructive content of a local-global principle
       
 with an application to the structure
 
 of a finitely generated projective module
 
 }

\smallskip 
--------------------------------

    {\Large \bf Le contenu constructif d'un principe local-global
    
 avec une application à la structure
 
 d'un module projectif de type fini
 
 }

\bigskip 
Henri Lombardi, Université de Franche-Comté, F-25030 Besançon Cedex, France, \\
email: {\tt Henri.lombardi@univ-fcomte.fr}

\end{center}

\bigskip 
This paper, written in French, appeared in\\ \emph{Publications Mathématiques de Besançon. Théorie des nombres.} Fascicule  94--95 \& 95--96, 1997.

\def\abstractname{\bf{Abstract}}
\begin{abstract}
We study the structure of an idempotent matrix $F$ over a commutative ring. We make explicit the fundamental system of orthogonal idempotents, hidden in this matrix, for each of which the matrix has a well-defined rank. Similarly we find a finite number of comaximal elements of the ring which make explicit the fact that the codomain of $F$ is locally free. Our proofs are based on the abstract local-global principle. We give two methods to recover a constructive proof of these results. The most interesting one is a constructive interpretation of a very simple version of the abstract local-global principle. We think we have made a significant step towards a constructive version of the ``Hilbert programme’’ for abstract algebra, i.e. the \emph{automatic} translation of proofs of abstract algebra into constructive proofs.
\end{abstract}

\smallskip 
\centerline{--------------------------------}

\def\abstractname{\bf{Résumé}}
\begin{abstract}
Nous étudions la structure d'une matrice projecteur $F$ sur un anneau 
commutatif. Nous explicitons le système fondamental d'idempotents 
orthogonaux, caché dans cette matrice, pour chacun desquels la matrice a 
un rang bien défini. De même nous trouvons un nombre fini 
d'éléments de l'anneau qui l'engendrent en tant qu'idéal et qui 
permettent d'expliciter le module projectif image de $F$ comme localement 
libre. Nos démonstrations sont basées sur le principe local-global abstrait. 
Nous donnons deux méthodes pour récupérer une démonstration constructive 
des résultats obtenus. La plus intéressante est une interprétation 
constructive du principe local-global abstrait le plus élémentaire. Il 
nous semble qu'il s'agit là d'un pas non négligeable dans la mise en 
place du ``programme de Hilbert" pour l'algèbre abstraite, i.e. la 
traduction {\em automatique }  des démonstrations d'algèbre abstraite en 
démonstrations constructives. 
\end{abstract}

\smallskip 
\centerline{--------------------------------}

\bigskip This article was first rejected by the journal \emph{L’Enseignement Mathématique} without having been submitted to a reviewer. Perhaps the editor-in-chief felt that this article did not give a new interesting point of view concerning old results. More likely he was taken aback by the fact that the author claimed to revive Hilbert's programme, at least for abstract algebra.

 For the author, it was the confirmation of the efficiency of the dynamic method introduced in the article [2] which was only a pre-publication at the time. In the article [2] the method has been successfully used with complete formal first order theories in order to obtain constructive versions of the Nullstellensatz and its ordered and valued variants. In the present article, on the contrary, the corresponding formal theories are any 
first order coherent theories.

\newpage \thispagestyle{empty}
 In the book [9] quoted below, Theorems 1 and 2 of the present article are Theorems V-1.1 and V-1.3. The proofs come directly from those given here, but the reference to logic and dynamic theories has disappeared.

\bigskip 
\noindent {\bf References}

\medskip  Reference [2]  is today the following one.\\
Dynamical method in algebra: Effective Nullstellensätze. 
M. Coste,  H. Lombardi, M.-F. Roy.  
\emph{Annals of Pure and Applied Logic}
{\bf 111}, (2001) 203--256. 

\medskip Reference [10] cited as ``en préparation’’ is the paper  ``Platitude, localisation et anneaux de Prüfer,  
une approche constructive’’. 64 pages.  H. Lombardi, C. Quitté. In
\emph{Publications Mathématiques de Besançon.
Théorie des nombres.} Années 1998-2001.

\medskip Reference [9] has finally led to the book
{\em Algèbre commutative, Méthodes constructives}. 
 H.~Lombardi, C. Quitté. 2021, Calvage \& Mounet, seconde édition, revue et augmentée, du livre paru sous le même titre en 2011.

 \noindent An English translation appeared as {\em Commutative Algebra, Constructive Methods}\\ 
 H. Lombardi, C. Quitté. July 2015, Springer.\\
Revised version on ArXiv: \url{http://arxiv.org/abs/1605.04832}  



\bigskip 
\centerline{--------------------------------}


\bigskip L’article a d’abord été refusé par la revue \emph{L’Enseignement Mathématique} sans avoir été soumis à un rapporteur. Sans doute l’éditeur en chef estimait que cet article ne donnait pas un point de vue nouveau intéressant concernant des résultats anciens. Plus probablement il a été inquiété par le fait que l’auteur prétendait redonner vie au programme de Hilbert, au moins pour l’algèbre abstraite.   

  Pour l’auteur c’était la confirmation de l’efficacité de la méthode dynamique introduite dans l’article~[2] qui n’était à l’époque qu’une prépublication. Dans l’article [2] la méthode était utilisée avec succès avec des théories formelles complètes en vue d’obtenir des versions constructives du Nullstellensatz et de ses variantes ordonnées ou valuées. Dans l’article présent au contraire, les théories correspondantes sont des théories formelles cohérentes arbitraires.

 Dans le livre [9] cité ci-dessous, les théorèmes 1 et 2 de l’article présent sont les théorèmes V-1.1 \hbox{et V-1.3}. Les démonstrations sont directement issues de celles données ici, mais la référence à la logique et aux théories dynamiques a disparu. 

\bigskip \noindent {\bf Références}

\medskip  La référence [2]  est actualisée comme suit.\\
{\em Dynamical method in algebra: Effective Nullstellensätze.} 
M. Coste,  H. Lombardi et M.-F. Roy.  
Annals of Pure and Applied Logic
{\bf 111}, (2001) 203--256. 

\medskip La référence [10] citée ``en préparation’’ est l’article {\it Platitude, localisation et anneaux de Prüfer,  
une approche constructive}. 64 pages.  H. Lombardi et C. Quitté.
Publications Mathématiques de Besançon.
Théorie des nombres. Années 1998-2001.

\medskip La référence [9] a finalement abouti au livre
{\em Algèbre commutative, Méthodes constructives}, 
 H. Lombardi et C. Quitté, 2021. Calvage \& Mounet, seconde édition, revue et augmentée, du livre paru sous le même titre en 2011.
 
 \noindent Une traduction anglaise {\em Commutative Algebra, Constructive Methods}, 
 H. Lombardi, C. Quitté,  a été publiée par Springer en juillet 2015.\\
Une version révisée se trouve sur ArXiv: \url{http://arxiv.org/abs/1605.04832}

\newpage
\setcounter{page}{1}
\thispagestyle{empty}

%
%
%
%

\clearpage

\selectlanguage{french}
\def\frenchproofname{\textsl{Démonstration}}
\FrenchFootnotes
%
%


\pagestyle{headings}
\patchcmd{\sectionmark}{\MakeUppercase}{}{}{}

\newtheorem{theorem}{Théorème}   
\newtheorem{thdef}[theorem]{Théorème et définition}
\newtheorem{proposition}{Proposition}[section]
\newtheorem{lemma}[proposition]{Lemme}
\newtheorem{corollary}[proposition]{Corollaire} 
\newtheorem{propdef}[proposition]{Proposition et définition}
\newtheorem{plcc}[proposition]{Principe local-global concret}
\newtheorem{plca}[proposition]{Principe local-global abstrait}
\newtheorem{plcd}[proposition]{Principe local-global dynamique}
\newtheorem{fact}[proposition]{Fait}

\theoremstyle{definition}
\newtheorem{definition}[proposition]{Définition}
\newtheorem{convention}[proposition]{Convention}
\newtheorem{notation}[proposition]{Notation} 

\theoremstyle{remark}
\newtheorem{remark}[proposition]{Remarque}
\newtheorem{comment}[proposition]{Commentaire}
\newtheorem{example}[proposition]{Exemple}
\newtheorem{problem}[proposition]{Problème}
\newtheorem{question}[proposition]{Question}

\newcommand {\junk}[1]{}

\def \eop {\hbox{}\nobreak\hfill
\vrule width 1.4mm height 1.4mm depth 0mm \par \goodbreak 
\smallskip}

\def\.@{\char'76}
\def\mapright#1{\smash{\mathop{\longrightarrow}\limits^{#1}}} 
\def\maprightto#1{\smash{\mathop{\longmapsto}\limits^{#1}}} 
\def\mapdown#1{\downarrow\rlap{$\vcenter{\hbox{$\scriptstyle 
#1$}}$}}

\def\eqdefi{\buildrel{\rm def}\over =}

\def\eqdf#1{\buildrel{#1}\over =}
\def\vers#1{\buildrel{#1}\over \longrightarrow }
\def\impdef#1{\buildrel{#1}\over \Longrightarrow} 

\def \Z{\mathbb{Z}}
\def \N{\mathbb{N}}
\def \R{\mathbb{R}}

\def\={\hbox{~{\bf =}~}}
 
\def \E{\mathaccent19}    

\def \ni {\noindent}
\def \ss {\smallskip}
\def \sni {\ss\ni}
\def \ms {\medskip}
\def \mni {\ms\ni}
\def \bs {\bigskip}
\def \bni {\bs\ni}
\def \snic#1 {\smallskip\noindent\centerline{$#1$}\smallskip}
\def \ce{\centering}

\def \kXn {k[X_1,\ldots,X_n]}

\def \Bn {{\bf B}_n}

\def \cd {{\cal D}}
\def \ci {{\cal I}}
\def \cj {{\cal J}_n}
\def \cf {{\cal F}}
\def \cp {{\cal P}}
\def \cm {{\cal M}}

\def \I {{\rm I}}
\def \In {{\rm I}_n}
\def \Izero {{\ci}_{=0}}
\def \Ua {{\rm Unit}}
\def \Rzero {R_{=0}}
\def \Rua {R_{\Ua}}
\def \Mua {{\cm}_{\Ua}}
\def \Ja {{\rm Rnul}}
\def \Rja {R_{\Ja}}
\def \Ija {{\ci}_{\Ja}}

\def \Hom {{\rm Hom}}
\def \det {{\rm det}}
\def \Ker {{\rm Ker}}
\def \Im {{\rm Im}}
\def \Id {{\rm Id}}
\def \Mat {{\rm Mat}}
\def \Ann {{\rm Ann}}
\def \Spec {{\rm Spec}}

\def \Ap {A_{\cal P}}
\def \zg {\Z[G]}

\def \num {{n$^{{\rm o}}$}}

\def \sign {\mbox{\rm sign}}

\def \eop {\hbox{}\nobreak\hfill
\vrule width 1.4mm height 1.4mm depth 0mm \par \smallskip}

\title{ Le contenu constructif d'un principe local-global 
\\ avec une application~à la structure \\ d'un module projectif de type fini }
\author{ H. Lombardi }
\date{6 Janvier 97}
\maketitle

\centerline {Laboratoire de Mathématiques de Besan\c con}
\centerline {URA CNRS 741}
\centerline {UFR des Sciences et Techniques}
\centerline {Université de Franche-Comté}
\smallskip
\centerline {email: henri.lombardi@univ-fcomte.fr}
\bigskip
\bigskip

\begin{abstract}
Nous étudions la structure d'une matrice projecteur $F$ sur un anneau 
commutatif. Nous explicitons le système fondamental d'idempotents 
orthogonaux, caché dans cette matrice, pour chacun desquels la matrice a 
un rang bien défini. De même nous trouvons un nombre fini 
d'éléments de l'anneau qui l'engendrent en tant qu'idéal et qui 
permettent d'expliciter le module projectif image de $F$ comme localement 
libre. Nos démonstrations sont basées sur le principe local-global abstrait. 
Nous donnons deux méthodes pour récupérer une démonstration constructive 
des résultats obtenus. La plus intéressante est une interprétation 
constructive du principe local-global abstrait le plus élémentaire. Il 
nous semble qu'il s'agit là d'un pas non négligeable dans la mise en 
place du ``programme de Hilbert'' pour l'algèbre abstraite, i.e. la 
traduction {\sl automatique}  des démonstrations d'algèbre abstraite en 
démonstrations constructives. 
\end{abstract}

\bni {\bf Classification AMS.} 03F65, 13C10, 13B10

\bni {\bf Mots clés.}  Mathématiques constructives, Programme de
Hilbert, \'Evaluation dynamique, Mo\-du\-les projectifs de type fini,
Matrices de projection, Idéaux de Fitting, Principes local-globals.

\bigskip \noindent Note. Dans la version présentée ici, le texte est inchangé mais on a corrigé des fautes d'orthographe, amélioré la mise en page, actualisé l'adresse mail et ajouté une table des matières.

\newpage 
\tableofcontents
\newpage

\markboth{Introduction}{Introduction}
\section*{Introduction} \label{sec Introduction}
\addcontentsline{toc}{section}{Introduction}
Dans cet article, tous les anneaux considérés sont commutatifs.

Notre but est de comprendre en termes concrets les théorèmes suivants.
\begin{theorem}[caractérisation locale des modules projectifs de type fini] \label{th ptf-loc} ~

\noindent  Un module $M$  sur un anneau~$A$  est 
projectif de type fini si et seulement si il est {\em  localement libre} 
au sens suivant: il existe $s_1,\ldots,s_m$ dans~$A$ tels que,

\noindent $\bullet$  $s_1A+\cdots +s_mA=A$, et

\noindent $\bullet$  les $M_{s_i}$ obtenus~à partir de $M$ en étendant 
les scalaires aux  $A_{s_i}$ ($A_s$ désigne le localisé où on 
autorise le dénominateur $s$)  sont libres.
\end{theorem}
\begin{theorem}[décomposition d'un module  projectif de type fini en somme directe de modules de rang constant] \label{th ptf-idpt} ~

\noindent  Si  $M$ est un module projectif de type fini sur un anneau~$A$ engendré par~$n$ éléments, il existe un système fondamental d'idempotents 
orthogonaux $r_0, r_1, \ldots, r_n$ tel que chaque localisé $M_{r_k}$ 
soit un module projectif de rang $k$ sur $A_{r_k}$. En outre~$A$ est 
naturellement isomorphe~à $A_{r_0}\times A_{r_1}\times\cdots\times 
A_{r_n}$ et $M$~à $M_{r_0}\times M_{r_1}\times\cdots \times M_{r_n}$. 
\end{theorem}

La partie la plus mystérieuse du théorème \ref{th ptf-loc} est que 
la condition est nécessaire. En pratique, le module $M$ peut être vu 
comme l'image dans  $A^n$ d'une {\sl  matrice de projection } $F$ (i.e., 
$F^2=F$)~à coefficients dans~$A$. On veut récupérer les $s_k$~à 
partir des coefficients  $f_{i,j}$ de $F$. De même dans le théorème 
\ref{th ptf-idpt} on veut récupérer les $r_k$~à partir des 
coefficients  $f_{i,j}$ de $F$.  Ceci est réalisé dans les 
théorèmes~\ref{th matproj} et~\ref{theorem loc libre}.

L'idée générale pour obtenir ces résultats est la suivante. On 
remarque pour commencer que si~$A$ est intègre, le module est de rang 
$k$ avec 
$0\leq k\leq n$ et le polynôme  caractéristique
 de $F$ est alors \hbox{$(X-1)^kX^{n-k}$}. Son coefficient de degré $n-k$ est 
égal~à $(-1)^k$ et c'est la somme des mineurs diagonaux d'ordre $k$. 
Ce sont ces mineurs qu'il faut prendre comme $s_i$ pour obtenir les 
localisés libres dans le théorème \ref{th ptf-loc}.
Enfin, si on ne suppose pas~$A$ intègre, et notamment dans le cas 
générique, les rangs possibles se mélangent de manière bien 
contrôlée gr\^ace~à un système fondamental d'idempo\-tents 
orthogonaux qui se lisent sur le polynôme caractéristique de $F$.

En fait 
on est particulièrement intéressé par le cas générique: 
$A=\Bn =\Z[(f_{i,j})_{1\leq i,j\leq n}]/\cj$, où~$\cj$ 
est l'idéal défini par les $n^2$ relations obtenues en écrivant 
$F^2=F$.

\ms Le principe local-global abstrait en algèbre commutative est un 
principe informel selon lequel certaines propriétés concernant les 
modules sur les anneaux commutatifs sont vraies si et seulemment si elles 
sont vraies après localisation en n'importe quel idéal premier. Dans 
nos démonstrations, le seul ingrédient non constructif est un principe 
local-global abstrait de recollement des égalités.  
La démonstration de ce principe utilise des outils hautement non constructifs 
(dont le recours~à la considération de tous les idéaux premiers de 
$A$). Dans la section \ref{sec local-global}, nous expliquons comment 
interpréter de manière constructive ce principe local-global. En gros, 
le cadre de l'évalutation dynamique permet de traiter les idéaux 
premiers de l'anneau comme des objets idéaux présents seulement~à 
l'état latent et parfaitement inoffensifs. Ceci nous permet 
d'interpréter l'utilisation du principe abstrait de recollement des 
égalités comme une machinerie purement calculatoire~à l'intérieur 
des évalutations dynamiques. En définitive, nous récupérons une 
démonstration constructive complète des théorèmes concrets que ce principe 
permet de démontrer. 

Il nous semble qu'il s'agit là d'un pas non négligeable dans la mise 
en place du ``programme de Hilbert'' pour l'algèbre abstraite, i.e. la 
traduction automatique des démonstrations d'algèbre abstraite en démonstrations 
constructives{\footnote{~Du moins lorsque le résultat est de nature 
concrète}}. Notre espoir est notamment d'obtenir une relecture 
constructive automatique du chapitre IV de \cite{Kun} concernant la 
méthode générale de passage du local au global.        

\ms Les références générales pour ce travail sont les suivantes. 
Dans \cite{MRR} on trouve une approche constructive des bases de 
l'algèbre. Les théorèmes cités ci-dessus, pour lesquels nous 
demandons une explicitation précise, ainsi que ceux cités dans la 
section suivante (Rappels) sont dans les traités classiques d'algèbre 
commutative (cf. par exemple \cite{Kni}, \cite{Bou}, \cite{Kun}, 
\cite{Nor}.) 
Plus précisément on peut trouver le théorème \ref{th ptf-loc} 
comme (partie du) théorème 1 dans \cite{Bou} chap. II {\S}5, ou comme 
(partie du) théorème 3.3.7 de \cite{Kni}, on peut trouver   le 
théorème \ref{th ptf-idpt} comme exercice 3 dans \cite{Bou} \hbox{chap. II 
{\S}5}.

\ms Nous n'avons pas trouvé dans la littérature concernant la 
structure des modules projectifs de type fini des théorèmes aussi 
explicites que les théorèmes \ref{theorem loc libre}, \ref{th ptf rang 
constant} et \ref{theorem rang constant}, que nous donnons~à la section~\ref{sec Matproj}. Il nous semble également que pour certains autres 
résultats de nature concrète contenus dans cet article, il n'existait 
pas pour le moment de démonstration entièrement constructive. Nous l'avons 
alors signalé dans le cours de l'article.  

\ms L'article est organisé comme suit. Dans la section 1, nous faisons 
quelques rappels d'algèbre commutative, dans le but notamment de mettre 
en valeur le caractère constructif de nombreux théorèmes de base et 
de présenter quelques aspects du principe local-global.

Dans la section \ref{sec Matproj}, nous donnons une explicitation précise des 
théorèmes \ref{th ptf-loc} et \ref{th ptf-idpt}. Nous faisons appel dans la démonstration~à un principe 
local-global abstrait élémentaire mais non constructif. Nous terminons 
en remarquant que dans le cas générique, toute la démonstration peut être 
rendue constructive, moyennnant un gros travail sur les idéaux des 
anneaux de polynômes~à coefficients entiers. Ceci assure la validité 
constructive de tous les théorèmes de la section \ref{sec Matproj}, dans tous les cas 
(pas seulement le cas générique).

Dans la section \ref{sec local-global}, nous donnons une interprétation constructive du 
principe local-global abstrait de recollement des égalités. Ceci 
permet de rendre constructives les démonstrations de la section \ref{sec Matproj} selon l'esprit 
du programme de Hilbert: donner une garantie automatique de la validité 
constructive des résultats concrets obtenus par des méthodes 
abstraites.

Dans la section \ref{sec iclg2}, nous apportons quelques compléments sur le thème du programme de Hilbert.

Dans l'article en préparation \cite{LQ}, nous donnons un traitement 
entièrement élémentaire, sans recours aux principes local-globals 
abstraits ni~à leur version dynamique et constructive, des résultats 
que nous démontrons ici. Dans un autre article en préparation 
(\cite{Lo}), nous essayons de tenir la promesse d'une relecture 
constructive automatique des principes local-globals abstraits dont nous 
avons connaissance.  

\mni {\bf  Remerciements.}  Nous remercions Fred Richman pour sa lecture 
attentive et ses suggestions.
\section{Rappels} \label{secRappels}
Nous donnons ici quelques rappels concernant les modules projectifs de 
type fini et la localisation, de manière~à faciliter la lecture de la 
suite au lecteur ou~à la lectrice non averti(e), et~à faciliter la 
discussion, dans la section \ref{sec local-global}, au sujet du 
caractère constructif des résultats obtenus précédemment. Le 
lecteur ou la lectrice{\footnote{~Désormais, la personne humaine qui 
intervient au cours de cet article subira la règle inexorable de 
l'alternance des sexes. Espérons que les lecteurs n'en seront pas plus 
affectés que les lectrices. En tout cas, cela nous éconmisera bien des 
``ou'' et bien des ``(e)''.}} 
 qui connaît bien ces sujets mais qui est intéressé(e) par la 
critique constructive des démonstrations classiques pourra donc également jeter 
un coup d’œil sur cette section. 
\subsection{Modules de présentation finie} 
\label{subsec prf}

Un module {\sl  de présentation finie} est un~$A$-module donné par un 
nombre fini de générateurs et de relations. De manière 
équivalente, c'est un module $M$ isomorphe au conoyau d'un homorphisme 
$$g: A^m\rightarrow A^q
$$
La matrice $G$ de $g$ a pour colonnes les relations entre les 
générateurs 
$a_1,\ldots,a_q$ (les images de la base canonique de $A^q$ par $\pi:A^n\to M\simeq \mathrm{Coker} (g)$).
Une telle matrice s'appelle une {\sl matrice de présentation du module~$M$}. 
On ne change pas la structure de $M$ lorsque l'on fait subir~à $G$ une des 
transformations suivantes:

\noindent --- ajout d'une colonne nulle, (ceci ne change pas le module des 
relations entre des générateurs fixés),

\noindent --- suppression d'une colonne nulle, sauf~à aboutir~à une 
matrice vide,

\noindent --- remplacement de $G$, de type $q\times m$, par $G'$ de type
$(q+1)\times (m+1)$ obtenue~à partir de $G$ en rajoutant une ligne nulle 
en dessous puis une colonne~à droite avec $1$ en position $(q+1,m+1)$, 
(ceci revient~à rajouter un vecteur parmi les générateurs, en 
indiquant sa dépendance par rapport aux générateurs précédents):
$$G\;\mapsto \;G'\;=\;
\pmatrix{ 
    G      &C           \cr 
    0_{1,m}&1     
}$$

\noindent --- opération inverse de la précédente, sauf~à aboutir 
à une matrice vide,

\noindent --- ajout~à une colonne d'une combinaison linéaire des 
autres colonnes, (ceci ne change pas le module des relations entre des 
générateurs fixés),

\noindent --- ajout~à une ligne d'une combinaison linéaire des autres 
lignes, (ceci revient~à changer le système générateur en rempla\c 
cant par exemple le générateur $a_q$ par un élément de la forme
$a_q-\Sigma_{i=1,\ldots,q-1} \;\lambda_i a_i$ sans  changer les autres 
générateurs),

\noindent --- permutation de colonnes ou de lignes,

\noindent --- multiplication d'une colonne ou d'une ligne par un 
élément inversible  (facultatif).

\smallskip On voit aisément que si $G$ et $H$ sont deux matrices de 
présentation d'un même module~$M$, on peut passer de l'une~à 
l'autre au moyen des transformations décrites ci-dessus. Un peu mieux: 
on voit que pour tout système générateur fini de $M$, on peut 
construire~à partir de~$G$, en utilisant ces transformations, une 
matrice de présentation de $M$ correspondant au nouveau système 
générateur. Notez aussi qu'un changement de base de $A^q$ ou $A^m$ 
correspond~à la multiplication de $G$ (à gauche ou~à droite) par une 
matrice inversible, et peut être réalisé par les opérations 
décrites précédemment. 

Un module libre de rang $k$ est présenté par une matrice colonne 
formée de $k$ zéros.

Il existe un cas facile où une matrice présente un module libre. 
Rappelons que deux matrices de même type $q\times m$ sont dites {\sl  
équivalentes} lorsque l'on passe de la première~à la seconde en 
multipliant la première,~à droite et~à gauche, par deux matrices 
inversibles.

\mni {\bf  Lemme de la liberté.} \label{lem prf libre}  
{\sl Soit $M$ un module de présentation finie, (isomorphe au) conoyau 
d'une matrice~$G$ de type
$q\times m$ (i.e. le module est donné par $q$ générateurs soumis~à 
$m$ relations). Si la matrice $G$ contient un mineur d'ordre $k$ 
inversible et si tous les mineurs d'ordre $(k+1)$ sont nuls, alors elle 
est équivalente~à la matrice canonique
$$\I_{k,q,m}\;=\;
\pmatrix{ 
    \I_{k}   &0_{k,m-k}      \cr 
    0_{q-k,k}&     0_{q-k,m-k}      }$$
Alors, le module $M$ est libre de rang $q-k$. En fait, dans ce cas, 
l'image, le noyau et le conoyau de~$G$ sont libres, respectivement de 
rangs $k$, $m-k$ et $q-k$. En outre l'image et le noyau possèdent des 
supplémentaires libres.  
}
\begin{proof} En permutant éventuellement les lignes et les 
colonnes on ramène le mineur inversible en haut~à gauche. Puis en 
multipliant ~à droite  (ou~à gauche) par une  matrice inversible  on 
se ramène ~à la forme 
$$G_1\;=\; 
\pmatrix{ 
\I_k & A     \cr
  B &  C    \\
} $$ 
puis par des manipulations élémentaires de lignes et de colonnes, on 
obtient
$$G_2\;=\;
\pmatrix{ 
   \I_{k}   &0_{k,m-k}      \cr 
    0_{q-k,k}&     G_3}$$
et $G_3$ est nulle parce que tous les mineurs d'ordre $(k+1)$ de $G_2$ 
sont nuls.  
\end{proof}
\subsection{Modules projectifs de type fini} \label{subsec ptf}
Ils sont caractérisés de la manière suivante.
\begin{propdef}[modules projectifs de type fini] \label{propdef ptf}
Les propriétés suivantes pour un~$A$-module  $M$  sont 
équivalentes.
\begin{enumerate}
\item [a)]   $M$  est  isomorphe ~à un facteur direct dans un~$A$-module $A^n$, 
i.e. il existe un entier $n$, un~$A$-module $N$ et un isomorphisme 
$M\oplus N \rightarrow A^n$.
\item [b)]  Il existe un entier  $n$, des générateurs  $(g_i)_{i=1,\ldots 
,n}$  de  $M$  et des formes linéaires  $(\alpha_i)_{i=1,\ldots ,n}$  
sur  $M$  telles que :      $\quad \forall  x\in M   \qquad    x = 
\Sigma\; \alpha_i (x)  g_i$.
\item [b')]  $M$  est de type fini et pour tout système fini de générateurs  
$(h_i)_{i=1,\ldots ,m}$  de  $M$  il existe des formes linéaires  
$(\beta_i)_{i=1,\ldots ,m}$  sur  $M$  telles que:      $\quad \forall  
x\in M   \qquad    x = \Sigma\;  \beta_i (x)  h_i$.
\item [c)]   Il existe un entier  $n$  et deux applications linéaires   $\varphi 
: M\rightarrow A^n$  et  $\psi : A^n\rightarrow M$  telles que  $\psi\circ 
\varphi= {\rm  Id}_M$. On a alors 
$A^n=\Im(\varphi )\oplus\Ker(\psi)$  et $M\simeq\Im(\varphi )$.
\item [c')]   $M$  est de type fini et pour toute application linéaire 
surjective   $\psi : A^m\rightarrow M$ il existe une application
 linéaire $\varphi : M\rightarrow A^m$  telle que  
$\psi\circ \varphi= {\rm  Id}_M$.  On a alors 
$A^m=\Im(\varphi )\oplus\Ker(\psi)$  et $M\simeq\Im(\varphi )$.
\end{enumerate}
Lorsque ces conditions sont réalisées on dit que le module $M$ est 
{\em  projectif de type fini}.
\end{propdef}
\begin{proof} Le point (b) (resp (b')) n'est qu'une reformulation du point (c) (resp. (c')).

(a) $\Rightarrow $ (c): considérer les applications canoniques 
$M\rightarrow M\oplus N$ et  $M\oplus N\rightarrow M$.

(c) $\Rightarrow $ (a): considérer $\theta=\varphi\circ\psi$. On a 
$\theta^2=\theta$. Cela fournit la projection de $A^n$ sur $M$ 
parallèlement~à~$N$.

(b) $\Rightarrow $ (b'): en exprimant les $g_i$ comme combinaisons 
linéaires des $h_j$ on obtient les $\beta_j$~à partir des~$\alpha_i$.
\end{proof}

Une matrice de projection est une matrice carrée $F$ vérifiant $F^2=F$.
En pratique, conformément au (a) ci-dessus, nous considèrerons un 
module projectif de type fini comme (copie par isomorphisme de l') image 
d'une matrice de projection $F$.

Lorsqu'on voit un module projectif de type fini selon la définition (c), 
la matrice de projection est celle de l'application linéaire 
$\varphi\circ\psi$. De même, si on utilise la définition (b) la 
matrice de projection est celle ayant pour entrées les $\alpha_j (g_i)$ 
en position $(i,j)$.

\smallskip Si~$A$ est un anneau intègre, on obtient par passage au corps 
des fractions un espace vectoriel de dimension finie $k$. On en déduit 
que le polynôme caractéristique de la matrice $F$  est égal~à 
\hbox{$(X-1)^kX^{n-k}$} (nous considérons le polynôme caractéristique comme 
polynôme unitaire: $\det (X\In -F)$). Ceci caractérise en termes 
purement calculatoires la dimension $k$: le premier monôme non nul du 
polynôme caractéristique (en partant des bas degrés) est égal~à 
$(-1)^kX^{n-k}$. En outre tous les mineurs d'ordre~$k+1$ de $F$ sont nuls.

Ceci conduit~à la proposition suivante.
\begin{proposition}  \label{prop ptfrangconstant} 
Soit $k$ un entier naturel et $M$ un module projectif de type fini sur un 
anneau~$A$ non trivial. Alors les conditions suivantes sont 
équivalentes:
\begin{enumerate}
\item [a)] Pour tout idéal premier $\cp$ de~$A$, le module $M/\cp M$ sur 
l'anneau intègre $A/\cp$ est de rang $k$ (i.e. tout système de $k+1$ 
éléments est linéairement dépendant et il existe un système de 
$k$ éléments linéairement indépendant).
\item [a')] Pour tout idéal premier $\cp$ de~$A$, l'espace vectoriel obtenu~à 
partir de $M$ en étendant les scalaires au corps des fractions de 
$A/\cp$ est de dimension $k$.
\item [b)] Le polynôme caractéristique d'une matrice de projection $F$ de type 
$n\times n$ ayant pour image (un module isomorphe~à) $M$ est égal,~à 
des nilpotents près, au polynôme $(X-1)^kX^{n-k}$.
\item [b')] Même chose que (b), mais pour toute matrice $F$.
\item [c)] Le polynôme caractéristique d'une matrice de projection $F$ de type 
$n\times n$ ayant pour image (un module isomorphe~à) $M$ est égal,~à 
des nilpotents près, au polynôme $(X-1)^kX^{n-k}$, et tous les mineurs  
d'ordre $k+1$ de $F$ sont nilpotents.
\item [c')] Même chose que (c), mais pour toute matrice $F$.
\end{enumerate}
\end{proposition}
\begin{proof} D'un point de vue classique, la démonstration est 
immédiate~; il suffit de se rappeler que l'intersection des idéaux 
premiers est le nilradical de~$A$, i.e. l'ensemble des nilpotents. 
\\
Notez que d'un point de vue constructif, la condition (a) est a priori 
trop faible (par manque d'idéaux premiers), et les conditions (b) et (c) 
ne sont pas clairement équivalentes. 
\\
Une démonstration constructive de l'équivalence de (b) et (b') est une 
conséquence le lemme qui suit. 
\end{proof}

\begin{lemma} 
\label{lem polycar ptf}
Soient $F_1$ de type $m\times m$ et $F_2$ de type $n\times n$  deux
matrices de projection avec des images isomorphes. Alors on a
$$X^n\det(X\I_m-F_1)=X^m\det(X\I_n-F_2).
$$
\end{lemma}
\begin{proof}
On écrit $A^m\simeq M\oplus N_1$ et $A^n\simeq M\oplus N_2$ de sorte que 
$A^{n+m}\simeq M\oplus N_2\oplus M\oplus N_1$ on considère 
l'endomorphisme $f$ de $A^{n+m}$ qui est égal~à l'identité sur une 
composante 
$M$ et~à $0$ sur les trois autres composantes. Selon la manière
dont on regroupe les termes de la somme directe on trouve pour $f$
une ou l'autre des matrices
$$\pmatrix{ 
    F_1    &  0_{m,n}      \cr 
   0_{n,m} &  0_{n}     
}
\quad {\rm  ou} \quad 
\pmatrix{ 
    0_{m}    &    0_{m,n}      \cr 
    0_{n,m}  &    F_2     
}
$$
qui ont pour polynômes caractéristiques les deux membres de 
l'égalité~à démontrer.
\end{proof}
Ceci justifie constructivement la définition suivante.

\begin{definition} 
\label{def ptfrangconstant} 
Un module $M$ projectif de type fini sur un anneau~$A$ non trivial est dit 
{\sl  de rang constant égal~à $k$} lorsque la condition (b') de la 
proposition \ref{prop ptfrangconstant} est réalisée: le polynôme 
caractéristique d'une matrice de projection $F$ de type $n\times n$ 
ayant pour image (un module isomorphe~à) $M$ est égal,~à des 
nilpotents près, au polynôme $(X-1)^kX^{n-k}$.
\end{definition}

En fait, nous verrons plus loin  des caractérisations plus agréables 
des modules projectifs de rang constant. Notamment, on peut ``supprimer 
les nilpotents'' dans les conditions (b)--(c'). 
\begin{remark} 
\label{rem det ptf}
{\rm Avec une démonstration tout~à fait analogue~à celle du lemme \ref{lem 
polycar ptf}, on peut démontrer que le déterminant (et donc aussi le 
polynôme caractéristique) d'un endomorphisme d'un module projectif de 
type fini est bien défini{\footnote{~Bien que la démonstration du lemme 
\ref{lem polycar ptf} soit convaincante, il peut sembler un peu choquant 
que le déterminant d'un endomorphisme puisse être bien défini 
lorsque le rang du module lui-même n'est pas bien défini. 
Intuitivement, cela se passe comme suit: lorsque l'on décompose le module 
selon ses composantes équidimensionelles, chaque composante de 
l'endomorphisme a clairement un déterminant, et les déterminants en 
chaque dimension sont mis ensemble (via les idempotents correspondant aux 
composantes) pour former un déterminant global.}}. On peut alors lire la 
condition (b) comme signifiant que le polynôme caractéristique de 
l'endomorphisme $\Id_M$ est égal,~à des nilpotents près,~à 
$(X-1)^k$. 
} 
\end{remark}

\begin{convention} \label{conven rgc}
Lorsque l'anneau~$A$ est trivial (réduit~à $\{0\}$) tous les 
$A$-modules sont triviaux. Néanmoins, conformément~à la définition 
ci-dessus, il est logique de considérer que le module trivial est 
projectif de type fini de rang constant égal~à $k$, pour n'importe 
quelle valeur de l'entier $k\geq 0$. Cette convention permet une 
formulation plus uniforme des théorèmes et des démonstrations.
\end{convention}
\begin{definition} \label{def anneau local}
Un {\sl  anneau local} est un anneau où est vérifié l'axiome 
suivant:
$$\forall x\in A \qquad x \;\; {\rm  ou} \;\;  1-x\;\; {\rm est \; 
inversible. } $$
\end{definition} 

Notez que selon cette définition l'anneau trivial est local. Dans un 
anneau local, les éléments ``non inversibles''
 (ceux pour lesquels l'hypothèse d'inversibilité implique $1=0$ dans 
l'anneau~$A$) forment un idéal. Le quotient de l'anneau par cet idéal 
est un corps, appelé corps résiduel de l'anneau~$A$ (nous admettons 
l'anneau trivial comme corps). 
\begin{definition} \label{def discret} Un ensemble~$A$ muni d'une relation 
d'égalité est appelé {\sl  discret} lorsque l'axiome suivant est 
vérifié$$\forall x,y\in A \qquad x=y \;\; {\rm  ou} \;\;
  \lnot (x=y). $$
\end{definition}
\begin{comment} 
\label{comment discret}
{\rm  Classiquement, tous les ensembles sont discrets, car le ``ou'' 
présent dans la définition est compris de manière ``abstraite''. 
Constructivement, le ``ou'' présent dans la définition est compris 
selon la signification du langage usuel: une des deux alternatives au 
moins doit avoir lieu de manière certaine. Il s'agit donc d'un ``ou'' de 
nature algorithmique. Un ensemble est discret si on~a un test pour 
l'égalité de deux éléments arbitraires de cet ensemble. 
Constructivement l'ensemble des nombres réels n'est pas discret (plus 
précisément: le supposer discret impliquerait un principe 
d'omniscience qui n'est pas accepté constructivement, même si on ne 
peut prouver qu'un tel principe est absurde).\\
Le corps résiduel d'un anneau local est discret si et seulement si il y 
a un test d'inversibilité pour les éléments de~$A$. On dit dans ce 
cas que le groupe des unités $A^{\times}$ est une {\sl  partie 
détachable} de~$A$.
} 
\end{comment}

Rappelons que deux matrices carrées $m\times m$ sont dites {\sl  
semblables} lorsqu'elles représentent le même endomorphisme de $A^m$ 
sur deux bases distinctes (ou non). 

Nous donnons maintenant trois démonstrations différentes pour un lemme 
fondamental, que nous appelons lemme de la liberté locale.

\mni {\bf  Lemme de la liberté locale.}   \label{lelilo}
{\sl  
Soit~$A$ un anneau local. Tout module projectif de type fini sur~$A$ est 
libre. De manière équivalente: toute matrice de projection $F$ de 
type $n\times n$ est semblable~à une {\em matrice de projection 
standard}, c.-à-d. de la forme:}
$$\I_{k,n,n}=\pmatrix{ 
\I_k&   &0_{k,n-n}  \cr 
0_{n-k,k}&   &0_{n-k}      
}.$$
\begin{proof}[Première démonstration (classique usuelle)] Cette 
démonstration suppose que le corps résiduel est discret. À fortiori, on sait si 
l'anneau est trivial ou non. Si l'anneau est trivial, c'est clair. Si 
l'anneau est non trivial et si le corps résiduel est discret cela va 
aussi, en suivant la démonstration classique usuelle. Notons $\varphi 
:A^n\rightarrow A^n$ la projection de matrice $F$. On passe au corps 
résiduel, la matrice $F$ est alors la matrice de la projection sur le 
sous espace 
${\rm  Im}(\varphi )$ parallèlement au sous espace  ${\rm  Im}({\rm  Id} 
-\varphi )$. On considère alors un mineur résiduellement non nul 
d'ordre maximum $k$ dans $F$, et de même  un mineur résiduellement 
non nul d'ordre maximum $n-k$ dans $\I_n-F$. En mettant cote~à cote les 
$k$ colonnes de $F$ et les $n-k$ colonnes de $\I_n-F$ correspondant~à 
ces mineurs, on obtient une matrice~$Q$ qui est résiduellement 
inversible, donc inversible (car son déterminant est inversible). La 
matrice $G=QFQ^{-1}$ représente l'application linéaire $\varphi $ sur 
une nouvelle base dont les $k$ premiers vecteurs sont dans  ${\rm  
Im}(\varphi )$ et les $n-k$ derniers sont dans 
${\rm  Im}({\rm  Id} -\varphi )$. Puisque $\varphi^2=\varphi $ ceci 
implique que~$G$ est la matrice de projection standard sur le sous espace 
des $k$ premiers vecteurs de base parallèlement au sous espace des $n-k$ 
derniers.
\end{proof}
\begin{proof}[Deuxième démonstration  (par la platitude)] Cette 
démonstration suppose aussi que le corps résiduel est discret. C'est une démonstration 
un peu plus ``calculatoire'', qui sera plus facile~à utiliser dans la 
section~\ref{sec iclg2}. Nous l'avons extraite de la démonstration classique qui 
démontre d'abord qu'un module projectif est plat, puis qu'un module plat 
de présentation finie sur un anneau local est libre.
\\
Tout d'abord, nous établissons le lemme suivant:
\begin{lemma}[lemme de la présentation locale] 
\label{lem prf local} 
Soit  $A$ un anneau local dont le corps résiduel est discret. Une 
matrice $G$ de type $q\times m$~à coefficients dans~$A$ est 
équivalente (sur~$A$)~à une matrice:
$$\pmatrix{ 
    \I_{k}   &0_{k,m-k}      \cr 
    0_{q-k,k}&     G'      
}$$
où $G'$ a tous ses coefficients dans l'idéal maximal de~$A$.
\\
Tout module de présentation finie sur~$A$ peut être présenté par 
une matrice $G'$ de ce type.
\end{lemma}
\begin{proof}[Démonstration du lemme] On recopie, mutatis mutandis, la démonstration du 
lemme de la liberté. Notez que les matrices de passage $P$ et $Q$ se 
calculent explicitement~à partir de $G$ une fois qu'on a repéré un 
mineur d'ordre $k$ inversible, tous les mineurs d'ordre $k+1$ étant non 
inversibles.
\end{proof}
En appliquant le lemme précédent, on obtient un entier $k$, des 
matrices $P$, $Q$, $P_1$, $Q_1$ inversibles et~$H$ résiduellement nulle, 
avec
$$PFQ\;=\;
\pmatrix{ 
    \I_{k}   &0_{k,n-k}      \cr 
    0_{n-k,k}&     H      }\;,\quad QQ_1=\In\;,\quad PP_1=\In.
$$
On a 
$$(PFQ)(Q_1P_1)(PFQ)=(PF^2Q)=(PFQ),
$$
ce qui se réécrit, avec $(Q_1P_1)$ décomposée en blocs:
$$\pmatrix{\I_{k}&0_{k,n-k}\cr 0_{n-k,k}&H}
\pmatrix{B&C\cr D&E}
\pmatrix{\I_{k}&0_{k,n-k}\cr 0_{n-k,k}&H}
=\pmatrix{\I_{k}&0_{k,n-k}\cr 0_{n-k,k}&H},
$$
c.-à-d., tous calculs faits
$$\pmatrix{B&CH\cr HD&HEH}
=\pmatrix{\I_{k}&0_{k,n-k}\cr 0_{n-k,k}&H},
$$
Ainsi $B=\I_k$ et $HEH=H$, donc $(\I_{n-k}-HE)H=0_{n-k}$. Mais $HE$ a ses 
coefficeints dans l'idéal maximal, donc $\det(\I_{n-k}-HE)=1+j$ avec $j$ 
dans l'idéal maximal est inversible. Donc $(\I_{n-k}-HE)$ est 
inversible, et $H=0_{n-k}$. Ceci implique que l'image de $F$ est un module 
libre de rang $k$ puisque
$$F\;=\;P_1\;
\pmatrix{ 
    \I_{k}   &0_{k,n-k}      \cr 
    0_{n-k,k}&   0_{n-k}      }\;Q_1.
$$
En fait, on a même
$$ PFP^{-1}=PFP_1=PFQQ_1P_1=\pmatrix{ 
    \I_{k}   &0_{k,n-k}      \cr 
    0_{n-k,k}&     0_{n-k}      }
\pmatrix{ 
    \I_{k}   &C      \cr 
    D&     E      }=
\pmatrix{ 
    \I_{k}   &      C      \cr 
    0_{n-k,k}&     0_{n-k}      },
$$
et donc en posant
$$R:=\pmatrix{ 
    \I_{k}   &      C      \cr 
    0_{n-k,k}&     I_{n-k}      },
$$
on obtient
$$R^{-1}= \pmatrix{ 
    \I_{k}   &      -C      \cr 
    0_{n-k,k}&     I_{n-k}      } \quad {\rm  et } \quad
(RP)F(RP)^{-1}=\pmatrix{ 
    \I_{k}   &0_{k,n-k}      \cr 
    0_{n-k,k}&     0_{n-k}      }.
$$
\end{proof}
\begin{proof}[Troisième démonstration  (à la Azuyama)] Cette démonstration 
ne suppose pas le corps résiduel discret. Elle est la traduction 
matricielle de la démonstration du théorème d'Azuyama (Theorem III.6.2 dans 
\cite{MRR}), pour le cas qui nous occupe ici. Nous allons diagonaliser la 
matrice~$F$. La démonstration fonctionne avec un anneau local non 
nécessairement commutatif.\\ 
Appelons $f_1$ le vecteur colonne $f_{1..n,1}$ de la matrice $F$, et 
$e_1,\ldots ,e_n$ la base canonique de $A^n$.\\
\noindent -- Premier cas, $f_{1,1}$ est inversible. Alors $f_1,e_2,\ldots 
,e_n$ est une base de $A^n$. Par rapport~à cette base $\varphi$ a une 
matrice:
$$G:=\pmatrix{ 
    1   &      li      \cr 
    0_{n-1,1}&   F_1      }.$$
En écrivant $G^2=G$ on obient $F_1^2=F_1$ et $F_1li=0$.
On a alors:
$$LGL^{-1}:=\pmatrix{ 
    1   &      li      \cr 
    0_{n-1,1}&   \I_{n-1}      }
\pmatrix{ 
    1   &      li      \cr 
    0_{n-1,1}&   F_1      }
\pmatrix{ 
    1   &      -li      \cr 
    0_{n-1,1}&   \I_{n-1}      }= 
\pmatrix{ 
    1   &      0_{1,n-1}      \cr 
    0_{n-1,1}&   F_1      }.
$$

\noindent -- Deuxième cas, $1-f_{1,1}$ est inversible. Alors 
$e_1-f_1,e_2,\ldots ,e_n$ est une base de $A^n$. Par rapport~à cette 
base, $\Id_n-\varphi$ a une matrice:
$$G:=\pmatrix{ 
    1   &      li      \cr 
    0_{n-1,1}&   F_1      }$$
avec $G^2=G$. Avec le même calcul que dans le cas précédent, 
$\I_n-F$ est donc semblable~à une matrice
$$\pmatrix{ 
    1   &      0_{1,n-1}      \cr 
    0_{n-1,1}&   F_1      }
$$
avec $F_1^2=F_1$, ce qui signifie que $F$ est semblable ~à une matrice:
$$\pmatrix{ 
    0   &      0_{1,n-1}      \cr 
    0_{n-1,1}&   H_1      }
$$
avec $H_1^2=H_1$.\\
On termine la démonstration par récurrence sur $n$.
\end{proof}

\begin{comment} 
\label{comment ptf libre}
Du point de vue classique, tous les ensembles sont discrets, et 
l'hypothèse correspondante est superflue dans les deux premières 
démonstrations.  
Nous avons signalé les trois démonstrations parce que le lemme de la liberté 
locale est un lemme crucial dans la suite, et que les différentes 
démonstrations conduisent~à différentes méthodes, plus ou moins 
compliquées, permettant de rendre constructifs les théorèmes que 
nous avons en vue. 
\end{comment}

\subsection{Localisation} \label{subsec locali}

Nous supposons la lectrice familière du processus de localisation en une 
partie multiplicative~$S$ de~$A$, ainsi qu'avec les notations $A_S$, $M_S$ 
(pour le localisé du~$A$-module $M$), et~$A_s$, $M_s$ lorsque $S$ est 
engendré par l'élément $s$ de~$A$. Nous voulons cependant garder la 
possibilité de localiser en un monoïde (multiplicatif) pouvant 
contenir 0. Le résultat est alors l'anneau trivial (et le module 
trivial). 

Des résultats essentiels sont les suivants:
\begin{fact} \label{fact sexloc}~
\begin{enumerate}
\item Si $M$ est un sous module de $N$, on a l'identification canonique de $M_S$ 
avec un sous module de~$N_S$ et de $(N/M)_S$ avec $N_S/M_S$.
\item Si $f:M\rightarrow N$ est une application~$A$-linéaire, 
${\rm  Im}(f_S)$  s'identifie canoniquement~à $({\rm  Im}(f))_S$, 
${\rm  Ker}(f_S)$  s'identifie canoniquement~à $({\rm  Ker}(f))_S$ et 
${\rm  Coker}(f_S)$  s'identifie canoniquement~à $({\rm  Coker}(f))_S$.
\item 
Si $$M\vers{f}N\vers{g}P$$ est une suite exacte de~$A$-modules et 
$S\subset A$ un monoïde, alors  
$$M_S\vers{f_S}N_S\vers{g_S}P_S$$ 
est une suite exacte de $A_S$-modules.   
\end{enumerate}
\end{fact}
\begin{fact} 
\label{fact hom egaux}
Soit $f:M\rightarrow N$, $g:M\rightarrow N$ deux applications linéaires 
entre~$A$-modules, {\em  avec $M$ de type fini}. Soit $S$ un monoïde 
de~$A$.  Alors $f_S=g_S$ si et seulement si il existe $s\in S$ tel que  
$sf=sg$. En d'autres termes, l'application canonique 
$(\Hom_A(M,N))_S\rightarrow \Hom_{A_S}(M_S,N_S)$ est injective.
\end{fact}
\begin{fact} \label{fact homom loc prf} Soient $M$ et $N$ deux~$A$-modules, 
$S$ un  monoïde de~$A$ et $\varphi:M_S\rightarrow N_S$ une 
application $A_S$-linéaire. On suppose que  $M$ est de présentation finie. 

\ni Alors il existe une application~$A$-linéaire $\phi:M\rightarrow N$ 
et $s\in S$ tels que
$$\forall x\in M\quad  \varphi \Big({x\over 1}\Big) = {\phi(x)\over s} \;.$$
En d'autres termes, l'application canonique $(\Hom_A(M,N))_S\rightarrow 
\Hom_{A_S}(M_S,N_S)$ est bijective.
\end{fact}
\begin{proof} (Cf. \cite{Nor} exercice 9 p. 50 ou \cite{Kun} chap. IV proposition 1.10) 
Supposons que $M$ est le conoyau de l'application linéaire 
$g:A^m\rightarrow A^q$ avec une matrice $G=(g_{i,j})$ par rapport aux 
bases canoniques, alors d'après le fait \ref{fact sexloc} $M_S$ est le 
conoyau de l'application linéaire $g_S:A_S^m\rightarrow A_S^q$ avec la 
matrice $G_S=(g_{i,j}/1)$ par rapport aux bases canoniques.
On note $j_m:A^m\rightarrow A_S^m$, $j_q:A^q\rightarrow A_S^q$,  
$j_M:M\rightarrow M_S$, $j_N:N\rightarrow N_S$, $\pi:A^q\rightarrow M$, 
$\pi_S:A_S^q\rightarrow M_S$ les applications canoniques. Soit 
$\psi:=\varphi \circ \pi_S$, de sorte que $\psi\circ g_S=0$. 
Donc $\psi\circ g_S\circ j_m=0=\psi \circ j_q \circ g$. Il existe un 
dénominateur commun $s\in S$ pour les images par $\psi$ des vecteurs de 
la base  canonique, donc il existe une application linéaire 
$\Psi:A^q\rightarrow N$ avec $(s\psi)\circ j_q=j_N\circ \Psi$. D'où 
$j_N\circ\Psi\circ  g= s(j_m\circ g_S\circ \psi)=0$. D'après le fait 
\ref{fact hom egaux} appliqué~à $\Psi\circ g$,  l'égalité 
$j_N\circ(\Psi\circ  g)= 0$ dans $N_S$ implique qu'il existe $s'\in S$ tel 
que $s'(\Psi\circ  g)=0$. Donc $s' \Psi$ se factorise sous forme 
$\phi\circ \pi$. 
On obtient alors $(ss'\varphi)\circ j_M\circ \pi=ss'(\varphi 
\circ\pi_S\circ j_q)=ss'\psi\circ j_q=s'j_N\circ\Psi=j_N\circ \phi\circ 
\pi$, et puisque $\pi$ est surjective $ss'\varphi \circ j_M=j_N\circ\phi$. 
C.-à-d., pour tout $x\in M$ $\varphi(x/1)=\phi(x)/ss'$.
\begin{figure}[htbp]  
\centerline{
\xymatrix{
A^m \ar[d]_{\displaystyle g} \ar[rrr]^{\displaystyle j_m} 
& &
&A_S^m\ar[d]^{\displaystyle g_S} & \\
A^q \ar[d]_{\displaystyle \pi}\ar[rdd]^(.4){\displaystyle
\Psi}\ar[rrr]^{\displaystyle j_q} & & 
&A_S^q\ar[d]_{\displaystyle\pi_S}\ar[rdd]^{\displaystyle 
\psi} & \\
M\ar[rrr]
^{\displaystyle j_M}\ar[rd]_{\displaystyle \phi}
& & &M_S\ar[rd]_{\displaystyle \varphi} & \\
 & N\ar[rrr]^{\displaystyle j_N} & & &N_S}
}
\caption[]{\label{fig1}  Localisation des homomorphismes
}  
\end{figure}
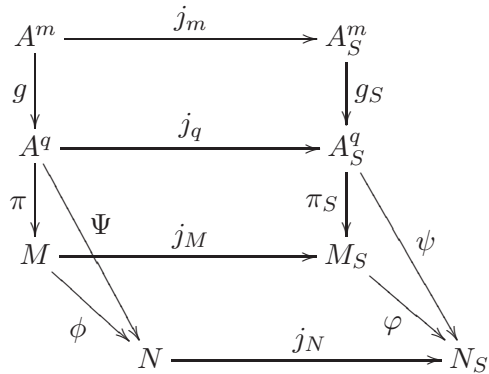  
\end{proof}

Un cas particulier est le suivant.

\begin{fact} \label{fact lin form loc prf} Soit $M$ un~$A$-module de 
présentation finie, $S$ un  monoïde de~$A$ et 
$\varphi:M_S\rightarrow A_S$ une forme $A_S$-linéaire.  
Alors il existe une forme~$A$-linéaire $\phi:M\rightarrow A$ et 
$s\in S$ tels que
$$\forall x\in M\quad  \varphi \Big({x\over 1}\Big) = {\phi(x)\over s} .$$
\end{fact}
\begin{fact} 
\label{fact bilocal}
Si $S\subset S'$ sont deux monoïdes de~$A$ et $M$ est un~$A$-module on a 
des identifications canoniques $(A_S)_{S'}\simeq A_{S'}$  et  
$(M_S)_{S'}\simeq M_{S'}$.
\end{fact}
\subsection{Système fondamental d'idempotents orthogonaux} \label{subsec 
sfio}
Dans la suite nous serons amenés~à considérer l'anneau localisé 
$A_r$  où $r$ est un idempotent, ainsi que le localisé $M_r$ pour un 
$A$-module $M$.   Il est bon de remarquer que $A_r$ s'identifie 
canoniquement~à l'idéal $rA$ muni de la structure d'anneau où $r$ 
est l'élément neutre de la multiplication. L'application canonique de 
$A$ vers $A_r$ identifié~à $rA$ est donnée par $x\mapsto rx$.
Quant~à $M_r$, il s'identifie naturellement~à $rM$ (avec l'application 
canonique $M\rightarrow rM,\;x\mapsto rx$).

Si $M$ est image d'une application linéaire $f:A^n
\rightarrow A^n$ de matrice $F$, le module  $M_r$ s'identifie aussi 
naturellement~à  l'image de l'application linéaire 
$f_r:A_r^n\rightarrow A_r^n$ ayant pour matrice la matrice $rF$ (lorsque l'on 
identifie $A_r$ avec $rA$). Ceci résulte du fait \ref{fact sexloc} 
modulo les identifications canoniques.

Rappelons que dans un anneau~$A$ un {\sl   système fondamental 
d'idempotents orthogonaux} (sfio) est une liste d'éléments de~$A$, 
$(r_1,\ldots ,r_n)$, qui vérifie
$$r_ir_j=0 {\rm\; \; si\;\;  } i\neq j, \quad
 {\rm et }\quad \Sigma\; r_i=1  $$
(nous ne réclamons pas qu'ils soient tous non nuls). Ceci implique que 
$r_i=r_i^2$ pour chaque $i$.

On obtient alors:
\begin{fact} \label{fact sfio} Si $(r_1,\ldots ,r_n)$ est un sfio d'un 
anneau~$A$, et si $M$ est un~$A$-module, on a:
$$\displaylines{A \simeq   A_{r_1}\times\cdots
\times A_{r_n}, \cr M =  {r_1}M\oplus\cdots\oplus {r_n}M
 \simeq  M_{r_1}\times\cdots\times M_{r_n} .}$$
\end{fact}
\subsection{Le principe local-global} \label{subsec loc glob}
Un outil essentiel en algèbre classique est la localisation en (le 
complémentaire d') un idéal premier. Cet outil est a priori difficile 
à utiliser constructivement parce qu'on ne sait pas fabriquer les 
idéaux premiers qui interviennent dans les démonstrations classiques, et dont 
l'existence repose sur l'axiome du choix. Cependant, on peut remarquer que 
ces idéaux premiers sont en général utilisés~à l'intérieur de 
démonstrations par l'absurde, et ceci donne une explication du fait que le 
recours~à ces objets ``idéaux'' pourra être contourné et même 
interprété constructivement dans la section \ref{sec local-global}.

Le principe local-global abstrait en algèbre commutative est un principe 
informel selon lequel certaines propriétés concernant les modules sur 
les anneaux commutatifs sont vraies si et seulemment si elles sont vraies 
après localisation en n'importe quel idéal premier. 

\ss Nous étudions maintenant quelques cas élémentaires où le 
principe local-global s'applique. 

Nous commen\c cons~à chaque fois par des versions concrètes en 
apparence plus faibles, mais qui s'avèreront bien utiles, au moins d'un 
point de vue constructif. Pour ces versions concrètes, la localisation 
n'est pas réclamée ``en n'importe quel idéal premier'' mais en un 
nombre fini d'éléments de~$A$ qui engendrent~$A$ en tant qu'idéal. 
En langage savant, dans un principe local-global concret on recouvre le 
spectre de l'anneau par un nombre fini d'ouverts, tandis que dans un 
principe local-global abstrait on voit le spectre comme l'ensemble de ses 
points.

\ss Nous disons qu'un élément $a$ de~$A$ est {\sl  non diviseur de 
zéro} si la suite 
$$0\vers{}A\vers{a .}A$$
est exacte. Autrement dit, on~a:
$$\forall b\in A\quad (ba=0 \; \Rightarrow \; b=0).$$
C'est seulement pour l'anneau trival que 0 est non diviseur de zéro.
\begin{plcc} \label{plcc ring} Soient $s_1,\ldots s_n\in A$ avec 
$s_1A+\cdots +s_nA= A$, et soit $a\in A$. Alors on~a les équivalences 
suivantes:

\sni $\bullet$ Recollement concret des égalités:

\snic{ 
a=0\;    \Leftrightarrow \; 
\forall i\in \{1,\ldots,n \}\;\; a/1= 0  \quad   dans \quad   A_{s_i}  } 

\sni $\bullet$ Recollement concret des non diviseurs de zéro:

\snic{ 
a\;  est\;  non\;  diviseur\;  de\;  z\E ero\;  dans\;   A\;    
\Leftrightarrow  \;      
\forall i\in \{1,\ldots,n \} 
\;\;a/1\;  est\;  non\;  diviseur\;  de\;  z\E ero\;  dans\;   A_{s_i}
} 

\sni $\bullet$ Recollement concret des inversibles:

\snic{ 
a  \;  est \;inversible  \;dans\;    
A\;    \Leftrightarrow  \;     
\forall i\in \{1,\ldots,n \}\; \;a/1 \; \;  est \;inversible  \;dans\;    
A_{s_i}  
} 

\end{plcc}
\begin{proof} Les conditions sont nécessaires en raison du fait 
\ref{fact sexloc}. Une vérification directe est d'ailleurs immédiate.\\
Pour prouver que les conditions sont suffisantes, nous traitons sans perte 
de généralité  le cas avec $n=2$ et $s_1=s,\; s_2=t,\; s+t=1$. 

\smallskip\noindent
Supposons d'abord que $a/1=0$ dans~$A_s$ et dans $A_t$. Pour un entier 
$h\le 0$ convenable on~a donc $s^ha=0=t^ha$ dans~$A$. Or 
$1=(s+t)^{2h}=us^h+vt^h$ pour $u$ et $v$ convenables dans~$A$. Donc 
$a=1a=us^ha+vt^ha=u\times 0+v\times 0=0$ dans~$A$.

\smallskip\noindent 
Supposons maintenant que $a/1$ soit non diviseur de zéro   dans~$A_s$ et 
dans $A_t$. Soit $b\in A$ avec $ab=0$ dans~$A$ donc aussi $ab/1=0$ dans 
$A_s$ et dans $A_t$. On a donc $b/1=0$ dans~$A_s$ et dans $A_t$, donc 
aussi dans~$A$.

\smallskip\noindent 
Supposons enfin que $a/1$ soit inversible   dans~$A_s$ et dans $A_t$. 
Soient donc $b,c\in A$ et un entier $k\geq 0$ avec $ab/s^k=1$ dans~$A_s$ 
et $ac/t^k=1$ dans $A_t$, i.e., pour un entier $p\geq 0$, $abs^p=s^{p+k}$ 
et  $act^p=t^{p+k}$ dans~$A$. Posons $h=p+k$ et comme ci-dessus 
déterminons $u$ et $v$ dans~$A$ tels que $us^h+vt^h=1$  dans~$A$. Alors 
$a\times (ubs^p+vct^p)= us^h+vt^h=1$ dans~$A$. \end{proof}

\begin{notation} 
\label{nota Spec(A)}
On note $\Spec(A)$ l'ensemble des  idéaux premiers de~$A$. 

\ni Pour $\cp\in\Spec(A)$ et $S=A\setminus\cp$ on note $\Ap $ pour $A_S$ 
(l'ambiguité entre les deux notations contradictoires $\Ap $ et $A_S$ 
est levée en pratique par le contexte).

\ni Si $x$ est un élément d'un~$A$-module $M$, nous notons \\
$$\Ann(x):=\{a\in A\; ;\; ax=0\}$$
l'idéal annulateur de $x$.  
\end{notation}
La relation étroite qui existe entre les localisés locaux d'un anneau 
$A$ et ses idéaux premiers est précisée dans le fait suivant.
\begin{fact} 
\label{fact localise local}
Un monoïde $S$ d'un anneau~$A$ est dit {\em  saturé} lorsque l'on~a 
l'implication
$$
\forall s, t \in A \;\;( st\in S \;\Rightarrow\; s\in S).
$$
Pour qu'un monoïde multiplicatif saturé $S$ fasse de $A_S$ un 
anneau local non trivial, il faut et suffit que~$S=A\setminus \cp$ où 
$\cp$ est un idéal premier.\\
Par ailleurs, tout homomorhisme $A\rightarrow B$ de~$A$ vers un anneau 
local $B$ se factorise de manière unique par $A_{\cp}$ où $\cp$ est 
l'image réciproque de l'idéal maximal de $B$.
\end{fact}

La version abstraite puissante du principe local-global concret 
précédent est la suivante.
\begin{plca} \label{plca ring} Soit $a\in A$.  Alors on~a les 
équivalences suivantes:

\sni
$\bullet$ Recollement abstrait des égalités:

\snic{ 
a=0 \;   \Leftrightarrow \; 
\forall \cp\in\Spec(A)\; \;a/1= 0  \quad   dans \quad   \Ap   
} 

\sni $\bullet$  Recollement abstrait des non diviseurs de zéro:

\snic{ 
a\;  est\;  non\;  diviseur\;  de\;  z\E ero\;  dans\;   A \;   
\Leftrightarrow  \; 
\forall \cp\in\Spec(A)\; \; a/1\;  est\;  non\;  diviseur\;  de\;  z\E 
ero\;  dans\;   \Ap   
} 

\sni 
$\bullet$ Recollement abstrait des inversibles:

\snic{ 
a  \;  est \;inversible  \;dans\;    
A\;    \Leftrightarrow  \;     
\forall \cp\in\Spec(A)\; \;a/1 \; est \;inversible \;dans\; \Ap
} 
\end{plca}
\begin{proof}[Démonstrations  (non constructives).]
Les conditions sont nécessaires en raison du fait~\ref{fact sexloc}. Une 
vérification directe est d'ailleurs immédiate.\\
Pour les réciproques, nous supposons sans perte de généralité que 
l'anneau~$A$ est non trivial. 

\sni {\sl Première démonstration.}\\
Supposons d'abord $a\not= 0$ dans~$A$, soit $\Ann(a)$ l'idéal annulateur 
de $a$, qui est un idéal strict, soit $\cp$ un idéal premier contenant 
$\Ann(a)$ et soit $S=A\setminus \cp$. L'ensemble $S\cap \Ann(a)$ est vide, 
donc $a/1\not= 0$ dans $A_S$.\\
On en déduit la deuxième réciproque comme dans le cas analogue du 
principe local-global concret \ref{plcc ring}.\\
Supposons enfin $a$ non inversible dans~$A$. Soit $\cp$ un idéal premier 
contenant $aA$ et soit $S=A\setminus \cp$. Alors $a/1$ est non inversible 
dans $A_S$.   

\sni {\sl Deuxième démonstration.}  (pour les cas $a=0$ et $a$ inversible)\\
Pour chaque idéal premier $\cp$ on peut trouver $s\notin \cp$ tel que 
$a/1$ est nul (resp. inversible) dans~$A_s$. Les ouverts correspondants 
$U_s=\{\cp\in\Spec(A);\;s\notin \cp\}$ recouvrent $\Spec(A)$, donc les $s$ 
correspondants engendrent~$A$ comme idéal, donc un nombre fini d'entre 
eux, $s_1,\ldots ,s_m$ engendrent~$A$ comme idéal. On peut donc faire 
appel au principe local-global concret correspondant.
\end{proof}

\begin{comment} 
\label{comment plga1}
{\rm  La deuxième démonstration montre bien le lien entre le principe 
local-global abstrait et le principe local-global concret. Cependant, il 
ne semble pas qu'elle puisse jamais être rendue constructive. La 
première démonstration n'est pas non plus ``en général'' constructive, mais 
il existe des cas où elle l'est. Il suffit pour cela que les conditions 
suivantes soient vérifiées.

\ni Dans le cas du recollement des égalités:

\ni --- l'anneau~$A$ est discret,

\ni --- pour tout $a\not= 0$ dans~$A$ on sait construire un idéal 
premier $\cp$ de~$A$ contenant $\Ann(a)$.

\ni Dans le cas du recollement des inversibles:

\ni --- l'ensemble des $a\in A$ inversibles est une partie détachable de 
$A$,

\ni --- pour tout $a\in A$ non inversible, on sait construire un idéal 
premier $\cp$ de~$A$ contenant $aA$.

\noindent C'est par exemple le cas lorsque~$A$ est une algèbre de 
présentation finie sur $\Z$ ou sur un corps ``pleinement factoriel'' 
(voir \cite{MRR}).  
} 
\end{comment}

En pratique, on peut comprendre le principe local-global abstrait 
\ref{plca ring}  sous la forme intuitive suivante: pour démontrer un 
théorème d'algèbre commutative dont la signification est qu'un 
certain élément d'un anneau commutatif~$A$ est nul, non diviseur de 
zéro, ou inversible, il suffit de traiter le cas où l'anneau est 
local. C'est un principe du même genre que le principe de Lefschetz: 
pour démontrer un théorème d'algèbre commutative dont la 
signification est qu'une certaine identité algébrique a lieu, il 
suffit de traiter le cas où l'anneau est le corps des complexes (ou 
n'importe quel sous anneau qui nous arrange, d'ailleurs).

\smallskip Un résultat local-global concret, qui donne la moitié la 
plus facile du théorème 1, est le suivant.

\begin{plcc} 
\label{prop ptf-local-concret}[recollement concret de modules de type fini, de présen\-ta\-tion  finie ou projectifs de type fini] 

\noindent Supposons que $s_1,\ldots s_n\in A$ avec $s_1A+\cdots +s_nA= A$, et soit 
$M$ un~$A$-module.   Alors on~a les équivalences suivantes:

\noindent $\bullet$
$M$ est  de type fini si et seulement si chacun des $M_{s_i}$ est un 
$A_{s_i}$-module  de type fini.

\noindent $\bullet$
$M$ est de présentation finie si et seulement si chacun des $M_{s_i}$ 
est un $A_{s_i}$-module de présentation finie.

\noindent $\bullet$
$M$ est projectif de type fini si et seulement si chacun des $M_{s_i}$ est 
un $A_{s_i}$-module projectif de type fini.
\end{plcc}
\begin{proof} Les conditions sont clairement nécessaires. Pour 
prouver qu'elles sont suffisantes, nous traitons sans perte de 
généralité  le cas avec $n=2$ et $s_1=s,\; s_2=t,\; s+t=1$. 

Tout d'abord supposons que  $M_s$ est un~$A_s$-module  de type fini et 
$M_t$ est un $A_t$-module  de type fini. Montrons que $M$ est de type fini.
Soit $g_1,\ldots ,g_q$ des éléments de $M$ qui engendrent $M_s$ et 
$M_t$. Soit $x\in M$ arbitraire. On a pour un certain exposant $m$ et 
certains éléments  $a_1,\ldots ,a_q$ de~$A$ une égalité
$$ s^mx=a_1g_1+\cdots+a_qg_q \quad {\rm  dans} \quad M_s,   
$$
et donc pour un certain exposant $p$
$$ s^{m+p}x=s^pa_1g_1+\cdots+s^pa_qg_q \quad {\rm  dans} \quad M.   
$$
On écrit une égalité du même style avec $t$, et on les combine 
selon la procédure $us^h+vt^h=1$  comme dans les démonstrations précécentes.

\smallskip Supposons maintenant que  $M_s$ est un~$A_s$-module  de 
présentation finie et $M_t$ est un $A_t$-module  de présentation 
finie. Montrons que $M$ est de de présentation finie. 
\\
Soit $g_1,\ldots ,g_q$ un système générateur de $M$. 
\\
Soit $(a_{i,1},\ldots,a_{i,q})\in A_s^q$ des relations entre les $g_i/1\in 
M_s$ (i.e., $\Sigma_j \;a_{i,j}g_j=0$ dans $M_s$) pour\hbox{ $i=1,\ldots,k_1$}, 
qui engendrent le~$A_s$-module (contenu dans $A_s^q$) des relations entre 
les $g_j/1$. On peut supposer sans perte de généralité que chaque 
$a_{i,j}$ est en fait un élément $a'_{i,j}/1$ avec $a'_{i,j}\in A$. Il 
existe alors un exposant $n$ convenable tel que les vecteurs 
$s^n(a'_{i,1},\ldots,a'_{i,q}) = (a''_{i,1},\ldots,a''_{i,q})\in A^q$ 
soient des~$A$-relations entre les $g_j\in M$.
\\
Considérons de la même manière un système générateur de 
relations $(b_{i,1},\ldots,b_{i,q})\in A_t^q$ \hbox{(où $i=1,\ldots,k_2$)} 
entre les $g_i/1\in M_t$, avec $b_{i,j}=b'_{i,j}/1$ où 
$a'_{i,j}\in A$, puis 
$t^m(b'_{i,1},\ldots,b'_{i,q})=(b''_{i,1},\ldots,b''_{i,q})\in A^q$ qui 
sont des~$A$-relations entre les $g_j\in M$.
\\
Montrons que les deux systèmes de relations ainsi construits entre les
$g_j$ engendrent toutes les relations. Soit en effet une relation 
arbitraire $(c_1,\ldots ,c_q)$  entre les $g_j$. 
Considérons la comme une relation entre les $g_j/1\in M_s$ et 
écrivons la en conséquence comme combinaison~$A_s$-linéaire des 
vecteurs $(a''_{i,1},\ldots,a''_{i,q})\in A_s^q$. Après multiplication 
par une puissance convenable $s^{h}$ de $s$ on obtient une égalité 
dans $A^q$:
$$
s^{h}(c_1,\ldots ,c_q)=e_1(a''_{1,1},\ldots,a''_{1,q})+
\cdots+e_q(a''_{k_1,1},\ldots,a''_{k_1,q}). 
$$
On fait de même avec $t$ et il reste~à combiner les deux résultats 
selon la procédure $us^{h}+vt^{h}=1$  comme dans les démonstrations 
précécentes.

\smallskip Supposons enfin que  $M_s$ est un~$A_s$-module  projectif de 
type fini et $M_t$ est un $A_t$-module  projectif de type fini. Montrons 
que $M$ est  projectif de type fini. 
Puisque $M_s$ est projectif de type fini, il existe des formes 
$A_s$-linéaires $\alpha_1,\ldots ,\alpha_q$ sur $M_s$ telles que
$$ \forall x\in M_s \quad x=\alpha_1(x)g_1+\cdots+\alpha_q(x)g_q \quad 
{\rm  dans} \quad M_s.
$$
D'après le fait \ref{fact lin form loc prf}, puisque $M$ est de 
présentation finie, il existe un exposant $m$ et des formes  
$A$-liné\-aires $\alpha'_1,\ldots ,\alpha'_q$ sur $M$ telles que
$$ \forall x\in M \quad s^m\alpha_1(x)=\alpha'_1(x),\ldots,
s^m\alpha_q(x)=\alpha'_q(x) \quad {\rm  dans} \quad M_s,
$$
et donc
$$ \forall x\in M \quad s^mx=\alpha'_1(x)g_1+\cdots+\alpha'_q(x)g_q \quad 
{\rm  dans} \quad M_s.
$$
Donc, comme $M$ est de type fini (voir le fait \ref{fact hom egaux}) il 
existe un exposant $p$ tel que
$$ \forall x\in M \quad 
s^{m+p}x=s^p\alpha'_1(x)g_1+\cdots+s^p\alpha'_q(x)g_q.
$$
On écrit une égalité du même style avec $t$, et on les combine 
selon la procédure $us^h+vt^h=1$  comme dans les démonstrations précécentes.
 \end{proof}
\begin{remark} 
\label{rem ptf-local-concret}
Les démonstrations sont toujours ``les mêmes''. Il existe un traitement 
un peu plus abstrait, s'appuyant sur la notion de module fidèlement plat 
qui permet de voir pourquoi. Voir par exemple \cite{Kni} proposition 2.3.5 
et lemme 3.2.3. L'exposé dans \cite{Kni} du principe de recollement 
concret des modules projectifs de type fini manque de peu une démonstration 
entièrement constructive. Dans~\cite{Kun} ce principe est l'objet de la 
règle 1.14 du chapitre IV, mais là aussi la démonstration n'est pas 
constructive.   
\end{remark}

La principe local-global concret \ref{prop ptf-local-concret} de 
recollement des modules projectifs admet la version abstraite suivante. 
Nous n'utiliserons pas ce résultat.  
\begin{plca}[recollement abstrait de modules projectifs]
\label{prop ptf-local-abstrait}~\\
Soit $M$ un~$A$-module.  Supposons que $M$ soit de présentation finie ou 
que $M$ soit de type fini et~$A$ intègre, alors 
$M$ est projectif de type fini si et seulement si les localisés 
$M_{\cp}$, pour tous les ${\cp}\in \Spec(A)$  sont libres.
\end{plca}
\begin{proof} (cf. \cite{Nor} chap. 2, théorème 14 p. 43 et 
exercice 10 p. 51, \cite{Kni}  théorème 3.3.7).

\ni Nous donnons une démonstration pour le cas d'un module de présentation 
finie, distincte de celles citées ci-dessus. Cette démonstration fonctionne 
comme la deuxième démonstration du principe local-global abstrait \ref{plca 
ring}. 

\ni Il faut montrer que la condition est suffisante. Dire qu'une matrice 
$G$ présente un module libre de rang $k$ revient~à dire qu'on peut 
passer de $G$~à une matrice nulle de type $k\times 1$ par une suite 
finie de transformations élémentaires décrites~à la section 
\ref{subsec prf}.

\ni Soit maintenant $\cp$ un idéal premier. Si ce que nous venons 
d'expliquer fonctionne pour le $A_{\cp}$-mo\-du\-le~$M_{\cp}$ et un certain 
entier $k$, cela fonctionne aussi pour le~$A_s$-module $M_s$ pour un 
$s\in A\setminus \cp$ convenable, ceci en vertu du nombre fini 
d'égalités dans $A_{\cp}$ mises en jeu lors de ces transformations 
élémentaires.  

\ni Il reste~à recouvrir $\Spec(A)$ par un nombre fini d'ouverts 
$U_{s_i}$ et~à faire appel au principe local-global concret de 
recollement des modules projectifs de type fini.  
\end{proof}

Les deux principes qui suivent (concret et abstrait) ne seront pas 
utilisés dans la suite de l'article. Les démonstrations sont analogues~à 
celles des principes \ref{plcc ring}. Le principe concret peut par exemple 
être trouvé dans le livre de Knight \cite{Kni}.
\begin{plcc}[recollement concret des suites 
exactes] \label{plcc sex} ~\\
Supposons que $s_1,\ldots s_n\in A$ avec $s_1A+\cdots +s_nA= A$, et soit 
$f:M\rightarrow N$ et $g:N\rightarrow P$ des applications 
$A$-linéaires entre~$A$-modules. Alors la suite

\snic{
M\vers{f}N\vers{g}P}

\ni est exacte si et seulement si les suites

\snic{
M_{s_i}\vers{f_{s_i}}N_{s_i}\vers{g_{s_i}}P_{s_i}}

\ni 
sont exactes pour $i\in \{1,\ldots,n \}$.
\end{plcc}
\begin{plca}[recollement abstrait des suites 
exactes] \label{plca sex}~\\
Soit $f:M\rightarrow N$ et $g:N\rightarrow P$ des applications 
$A$-linéaires entre~$A$-modules. Alors la suite

\snic{
M\vers{f}N\vers{g}P}

\ni  
est exacte si et seulement si les suites 

\snic{
M_{\cp}\vers{f_{\cp}}N_{\cp}\vers{g_{\cp}}P_{\cp}}

\ni 
sont exactes pour tous les $\cp\in\Spec(A)$.
\end{plca}
\section{Matrices de projection} \label{sec Matproj}

\subsection{Cas d'un anneau local} \label{subsec cas local}
\begin{proposition}[cas d'un anneau local] \label{prop cas local}
Soit~$A$ un anneau local,  $F\in{\rm Mat}_n(A)$ avec $F^2=F$ et~$M$ le 
module projectif de type fini image de $F$ dans $A^n$. Il existe un  
entier $k$ $(0\le k \le n)$ tel que  
${\rm det}(\In+XF)=(1+X)^k$. En outre tous les mineurs d'ordre $k+1$ de 
$F$ sont nuls.
\end{proposition}
\begin{proof} Il s'agit d'une conséquence immédiate du lemme 
de la liberté locale: toute matrice de projection sur un anneau local 
est semblable~à une matrice de projection standard~$\I_{k,n,n}$. 
L'entier $k$ est  uniquement déterminé si l'anneau est non trivial.
\end{proof}

Notez que la démonstration précédente est entièrement constructive 
lorsqu'elle est basée sur la troisième démonstration du lemme de la liberté 
locale. Les deux autres démonstrations réclameraient que l'anneau local ait un 
corps résiduel discret.
\subsection{Cas général} \label{subsec cas general}
\begin{theorem}[matrices de projection: idempotents et localisations libres] 
\label{th matproj} 
~\\
Soit  $A$ un anneau, $F\in{\rm Mat}_n(A)$ avec $F^2=F$ et $M$ le module 
projectif de type fini image de $F$ dans~$A^n$. Posons 
$\R(1+X):={\rm det}(\In+XF),\,\R(X)=:r_0+r_1X+\cdots+r_nX^n$. 
Alors le système  
$(r_0,r_1,\ldots,r_n)$ 
est un système fondamental d'idempotents orthogonaux.

\ni En outre, les mineurs d'ordre $(k+1)$ de la matrice $r_kF$ sont tous 
nuls. Et si $s$ est un mineur diagonal d'ordre $k$ de $r_kF$, alors le 
module $M_s$ est libre de rang $k$ sur l'anneau~$A_s$. 
\end{theorem}
\begin{remark} \label{rem conv}
On notera que la dernière affirmation du théorème reste vraie 
si $s=0$ en raison de la convention \ref{conven rgc}. De même, avec 
cette convention la proposition \ref{prop cas local} reste vraie dans le 
cas d'un anneau trivial si on ne demande pas l'unicité de $k$. 
\end{remark}
\begin{remark} 
\label{rem polycar}
La définition des $r_i$ attachés~à la matrice $F$ peut être 
relue comme suit en utilisant le polynôme caractéristique sous sa forme 
usuelle:
$$\det(X\In-F)=:r_0X^n+r_1X^{n-1}(X-1)+\cdots+ r_iX^{n-i}(X-1)^i+
\cdots+r_n(X-1)^n$$
(les $X^{n-i}(X-1)^i$ forment une base du module des polynômes de degré 
$\leq n$, triangulaire par rapport~à la base usuelle) 
\end{remark}
\begin{proof}[Démonstration du théorème]
On a trivialement $\R(1)=1$, i.e. $\Sigma_{i}r_i=1$. On utilise le 
principe local-global abstrait de recollement des égalités pour 
montrer que $r_ir_j=0$ pour $i\neq j$. En effet cette égalité est 
vraie dans le cas des anneaux locaux d'après la proposition \ref{prop 
cas local} puisque tous les~$r_i$ sont égaux~à $0$, sauf un égal~à 
$1$.  

\ni La même astuce fonctionne pour démontrer que les mineurs d'ordre 
$k+1$ de $r_kF$ sont nuls. Soit en effet $t$ un mineur d'ordre $k+1$ de 
$F$, on doit montrer que $tr_k=0$ dans~$A$. Si~$A$ est local, ou bien 
$r_k=0$ (si le rang est $h\neq k$), ou bien $r_k=1$ et $t=0$ (si le rang 
est $k$). Voyons enfin la dernière affirmation. Nous notons $F_s$ la 
matrice $F$ vue dans~$A_s$. Il est clair que le mineur diagonal $s$ est 
inversible dans~$A_s$ et on vient de voir que tous les mineurs d'ordre 
$k+1$ de $r_kF\simeq F_{r_k}$ sont nuls, donc a fortiori tous les mineurs 
d'ordre $k+1$ de $F_s$ sont nuls. On peut donc
 appliquer le lemme de la liberté (page \pageref{lem prf libre}) et 
déduire que le~$A_s$-module $M_s$ image de la matrice $F_s$ est libre.  
\end{proof}
\begin{theorem}[forme explicite des théorèmes \ref{th ptf-loc} et \ref{th 
ptf-idpt}] \label{theorem loc libre} ~

\ni Sous les mêmes hypothèses et avec les mêmes notations qu'au 
théorème \ref{th matproj}, pour chaque $k=0,\ldots,n$, la matrice 
$r_kF$, vue comme matrice~à coefficients dans $A_{r_k}$ (identifié~à 
$r_kA$) a pour image un module 
projectif de rang $k$ sur l'anneau $A_{r_k}$ {\em  (ceci prouve le 
théorème \ref{th ptf-idpt})}. 

\ni Si les  $t_{k,i}$ sont les mineurs diagonaux d'ordre $k$  de  $F$, et 
si on pose $s_{k,i}=r_kt_{k,i}$, la somme (pour $k$ fixé) des $s_{k,i}$ 
est égale~à $r_k$, et chaque module  $M_{s_{k,i}}$ est libre de rang 
$k$. Donc la famille de tous les $s_{k,i}$ a pour somme $1$ et convient 
pour le théorème \ref{th ptf-loc}. 
En particulier, pour tout module projectif de type fini~à $n$ 
générateurs, $2^n$ éléments $s_i$ suffisent pour le théorème 
\ref{th ptf-loc}.
\end{theorem}
\begin{proof} Conséquence immédiate du théorème \ref{th 
matproj}.
\end{proof}
\begin{comment} 
\label{comment theorem loc libre}
{\rm  Le théorème précédent donne une version
complètement explicite des théorèmes~1 et 2. Nous sommes ici
dans une situation typique que se proposait de ``résoudre'' le
programme de Hilbert. Un énoncé explicite concret a été
démontré par des méthodes abstraites a priori peu fiables.
Nous donnons dans la suite deux moyens de récupérer une démonstration
entièrement fiable de l'énoncé concret. L'argument parfois
cité que tout théorème d'arihmétique prouvé dans ZFC peut
également être prouvé sans recours~à l'axiome du choix offre
au moins trois inconvénients. Le premier (mineur) est qu'une
analyse assez poussée doit être menée pour se convaincre qu'un
théorème comme le théorème~\ref{th ptf-idpt}~a, en fait, la
signification d'un théorème d'arithmétique. Le deuxième
(nettement plus sérieux) est que le recours~à l'axiome du choix
n'est pas le seul ingrédient non constructif dans la démonstration qui~a
été fournie. Le troisième (redoutable) est que rien ne
garantit que ZFC soit une théorie cohérente. } 
\end{comment}

Un autre corollaire du théorème \ref{th matproj} est le suivant.
\begin{theorem}[polynôme caractéristique des matrices de projection de rang 
constant] \label{th ptf rang constant} ~

\noindent 
Soit  $F\in{\rm Mat}_n(A)$ avec   $F^2=F$ et $M$ le module projectif de 
type fini image de $F$ dans $A^n$.
Alors le module $M$ est de rang $k$ si et seulement si le polynôme 
caractéristique de $F$ est égal~à $(X-1)^kX^{n-k}$. Dans ce cas tous 
les mineurs d'ordre $k+1$ de $F$ sont nuls.
\end{theorem}

\begin{proof} La condition est clairement suffisante. Montrons 
qu'elle est nécessaire. Nous supposons donc que le polynôme 
caractéristique de $F$ est égal,~à des nilpotents près, au 
polynôme $(X-1)^kX^{n-k}$. En appliquant le théorème \ref{th 
ptf-idpt}, cela implique que pour $h\neq k$ l'idempotent $r_h$ est 
nilpotent, donc nul. En ce qui concerne les mineurs d'ordre $k+1$ de $F$, 
on peut alors appliquer le théorème \ref{th matproj}.
\end{proof}

\begin{comment} 
\label{comment th ptf rang constant}
{\rm  Notez que dans la mesure où le théorème peut être prouvé 
constructivement, ceci nous donne une version constructivement 
satisfaisante de la proposition \ref{prop ptfrangconstant} (on ne 
considère pas le (a), et dans les autres conditions équivalentes, on 
peut évacuer les nilpotents). Nous verrons encore un peu mieux~à la 
section \ref{subsec relec const}.
} 
\end{comment}

Un dernier corollaire immédiat dans le même style est le suivant.
(cf. théorème 2 dans \cite{Bou} chap.~II~{\S}5).
\begin{theorem}[caractérisation locale des modules projectifs de rang 
constant] \label{theorem rang constant}~

\noindent  
Un~$A$-module $M$ engendré par $n$ éléments est projectif de rang 
constant $k$ si et seulement si 
il existe un entier $m\leq {n\choose k}$ et des éléments $s_1,\ldots 
,s_m$ de~$A$ tels que, d'une part $s_1A+\cdots +s_mA=A$, et d'autre part 
les modules $M_{s_i}$ soient libres de rang $k$.
\end{theorem}
Nous terminons cette section par une proposition facile.
\begin{proposition}[quand le localisé en un élément de~$A$ 
est de rang constant] 
\label{prop sfio unic}~

\noindent 
Soit $F$ une matrice de projection ayant pour image un module $M$, et 
$r_0,\ldots,r_n$ le sfio défini au théorème~\ref{th matproj}.\\ 
Soit $s$ un élément de~$A$. Pour que le localisé $M_s$ soit 
projectif de rang $h$ il faut et suffit que $r_h/1=1$ dans~$A_s$, 
c.-à-d. que $r_hs^m=s^m$ dans~$A$ pour un certain exposant $m$. Si $s$ 
est un idempotent, cela signifie que $r_h$ divise $s$. \\
Enfin si $s_0,\ldots,s_n$ est un sfio tel que chaque $M_{s_h}$ soit de 
rang $h$, alors $r_h=s_h$ pour $h=0,\ldots ,n$.
\end{proposition}

\subsection{Cas générique} \label{subsec cas generique}

Qu'est-ce que nous appelons le cas générique, concernant un module 
projectif~à $n$ générateurs~? 
On considère l'anneau $A=\Bn =\Z[(f_{i,j})_{1\leq i,j\leq n}]/\cj$, où 
$\cj$ 
est l'idéal défini par les $n^2$ relations obtenues en écrivant 
$F^2=F$. Dans cet anneau $\Bn$, nous avons la matrice $F=(f_{i,j})$ dont 
l'image dans $\Bn^n$ est ce qui mérite d'être appelé 
{\sl  le module projectif générique~à $n$ générateurs}.

Reprenons les notations du théorème \ref{th matproj} dans ce cas 
particulier. Dire que $r_hr_k=0$ dans $\Bn$ \hbox{(pour $0\le h\not= k\le n$)} 
signifie que, dans $\Z[{\bf  f}]=\Z[(f_{i,j})_{1\leq i,j\leq n}]$ 
$$r_h({\bf  f})r_k({\bf  f})
\in \cj\qquad (*)$$
Cela implique une identité algébrique qui permet d'exprimer cette 
appartenance. Cette identité algébrique est naturellement valable dans 
tous les anneaux commutatifs.
Il est donc clair que si l'appartenance $(*)$ est vérifiée dans le cas 
générique, elle implique  $r_hr_k=0$ pour n'importe quelle matrice de 
projection pour n'importe quel anneau commutatif.

La même chose vaut pour les égalités $r_hs=0$ lorsque $s$ est un 
mineur d'ordre $h+1$.

En résumé: si le théorème \ref{th matproj} est vérifié dans 
le cas générique, il est vérifié dans tous les cas.

Le seul ingrédient non constructif dans la démonstration du théorème 
\ref{th matproj} était l'appel au principe local-global abstrait 
\ref{plca ring}. Dans le commentaire après ce théorème, nous avons 
indiqué que le théorème admettait une démonstration constructive pour 
certains anneaux, en particulier pour les anneaux $\Z[x_1,\ldots,x_n]/\ci$ 
lorsque $\ci$ est donné comme un idéal de type fini.

Ainsi la démonstration classique est constructive dans le cas générique modulo
un gros travail sur les idéaux des anneaux $\Z[x_1,\ldots,x_n]$. 
Donc les théorèmes \ref{th ptf-loc}, \ref{th ptf-idpt}, \ref{th 
matproj}, \ref{theorem loc libre}, \ref{th ptf rang constant} et 
\ref{theorem rang constant} sont constructivement prouvés. 

\smallskip
Dans la section suivante, nous expliquons comment il est possible de 
suivre de beaucoup plus près la démonstration classique. Autrement dit encore, 
l'appartenance $(*)$ peut être  construite sans appel~à la (belle) 
théorie constructive de la noetherianité et des décompositions 
primaires pour l'anneau $\Z[x_1,\ldots,x_n]$.  
\section{Le contenu constructif du principe local-global} 
\label{sec local-global}

Notre but ici est donc de faire une relecture constructive de la démonstration du 
théorème~\ref{th matproj} dans le cas général (et non plus le cas 
générique) ``sans autres ingrédients algorithmiques que ceux 
contenus dans la démonstration classique’’. Cette affirmation quelque peu brutale 
ne doit pas être prise comme une boutade ni comme une provocation. 
Nous prétendons réellement débusquer un contenu algorithmique 
précis dans les {\sl  démonstrations} qui utilisent le principe local-global 
abstrait~\ref{plca ring}, même quand les idéaux premiers ne peuvent 
absolument pas être explicités en tant que tels. 

\subsection{L'idée générale} 
\label{subsec idgen}
Soit~$A$ un anneau commutatif et $a$ un élément de~$A$ qui est le 
résultat d'un certain calcul fait sous certaines hypothèses. Le 
principe local-global abstrait le plus élémentaire nous dit que~$a$ 
est nul dans~$A$ si et seulement si $a/1$ est nul dans tous les $\Ap$ 
(pour $\cp\in\Spec(A)$).

Supposons que nous ayions une démonstration que $a/1$ est nul dans tous les 
$\Ap$. Comme toute démonstration, elle est de nature finie. En particulier, 
l'{\sl  axiome des anneaux locaux}
$$\forall s,t\in A \qquad (\; s+t=1 \; \Rightarrow\;  s \;\; {\rm  ou} 
\;\; t\;\; {\rm est \; inversible }) $$
n'est utilisé qu'un nombre fini de fois dans la démonstration (cela est en 
relation étroite avec le théorème qui affirme que $\Spec(A)$ est 
quasicompact).

Au bout du compte la démonstration aura produit des éléments $s_1,\ldots 
,s_m$ de~$A$ qui vérifient $s_1A+\cdots +s_mA=A$ et pour lesquels $a/1$ 
est nul dans chaque $A_{s_i}$.

À condition d'être capable de suivre la démonstration de fa\c con 
suffisamment précise, on pourra donc conclure que $a=0$ dans~$A$ en 
utilisant cette fois-ci le principe local-global {\sl  concret}~\ref{plcc 
ring}.

À vrai dire, cette idée générale semble si simple et si
naturelle qu'il est étonnant qu'elle n'ait pas encore été
exploitée systématiquement. En fait, lorsque l'on essaie de mener
ce travail en détail, on voit apparaître un obstacle, c'est
que la plupart des démonstrations usuelles, même très simples, sont
néanmoins un peu trop compliquées pour pouvoir être
traitées directement selon l'idée générale précédente.
Par exemple, la démonstration classique usuelle du lemme de la liberté
locale utilise de manière cruciale le fait que le corps résiduel
est discret (cf. la première démonstration page~\pageref{lelilo}), ce qui
est un cas particulier d'usage du tiers exclu en logique classique. 

Il s'avère cependant que l'usage du tiers exclu n'est pas un
obstacle bien grave: l'usage de la logique classique est inoffensif
lorsqu'il s'agit de prouver des faits suffisamment concrets~! (cf.~\cite{clr} théorème 1.1). 

\subsection{Structures algébriques dynamiques} 
\label{subsec SAD}

Pour mettre en œuvre notre idée générale, nous aurons besoin de la 
notion de {\sl  structure algébrique dynamique} (cf. \cite{Lom}  et 
\cite{clr}).

L'idée qui gouverne la définition d'une structure algébrique 
dynamique est la suivante: il s'agit d'une structure algébrique 
incomplètement spécifiée, dans laquelle on calcule selon des 
règles de nature algébrique simple, celles qui définissent 
axiomatiquement une structure algébrique ordinaire. Le fait que la 
structure est incomplètement spécifiée introduit une arborescence 
dans les calculs.

Par exemple si on dit: voici un corps engendré par 2 éléments $a$ 
et $b$ qui vérifient \hbox{$a^2+b^2+1=0$}, les calculs qui s'ensuivent peuvent 
faire apparaître dans les différentes branches n'importe quelle 
situation correspondant~à cette ``présentation''. Dans un premier 
embranchement $a$ sera nul et dans un autre,~$a$ sera inversible, puisque 
tout élément dans un corps est nul ou inversible. D'autres 
embranchements peuvent apparaître si~à un moment donné du calcul, 
on se pose par exemple la question de savoir si 5 est nul ou inversible.

Autre exemple, qui nous concerne directement ici. Si on dit: voici un 
anneau~$A$ complè\-tement spécifié en tant qu'anneau, mais 
appliquons lui les règles de calcul valables dans les anneaux locaux, 
les calculs vont faire apparaitre des embranchements chaque fois qu'on~a 
besoin d'utiliser l'axiome des anneaux locaux. On est alors en train de 
calculer ce qui se passe dans les différents localisés~$\Ap$ de~$A$. 
Différents cas peuvent se produire: ils sont pris en compte dans les 
différentes branches du calcul. Si la démonstration aboutit, un nombre fini de feuilles seulement apparaitront dans l'arbre du calcul. Cela veut dire 
qu'on n'a pas eu besoin de construire vraiment des localisés~$\Ap$, mais 
seulement des localisés~$A_s$ (qui en général ne sont pas des 
anneaux locaux). En langage savant: on~a recouvert le spectre de~$A$ par 
un nombre fini d'ouverts $U_s=\{\cp\in\Spec(A);\;s\notin \cp\}$. La 
différence entre le point de vue classique et le point de vue 
constructif est alors seulement que le mathématicien classique ``admet'' 
que les idéaux~$\cp$ existent en vertu (d'une version faible) de 
l'axiome du choix, tandis que la mathématicienne constructive (qui ne 
croit qu'à ce qu'elle voit) veut bien ``faire comme si'' ils existaient, 
puisque la seule chose importante dans ce spectre, ce ne sont pas ses 
points, mais ses recouvrements ouverts finis.   

\smallskip
Tout ceci semble avoir quelque rapport avec les tableaux sémantiques en 
logique. Des rapports étroits existent également avec la théorie des 
topos cohérents (cf.~\cite{clr}) et avec la théorie des esquisses 
(cf.~\cite{dr1} et~\cite{dr2}). Notre première inspiration a été fournie 
par l'évaluation dynamique de la clôture algébrique d'un corps ``à 
la D5'' (cf. \cite{ddd}) qui réalisait le fait remarquable suivant: 
{\sl  calculer de manière s\^ure dans la clôture algébrique d'un corps 
arbitraire alors même que cette clôture algébrique ne peut pas être 
construite (pour un corps général)}.
\subsection{Anneau versus anneau local (dynamiques)} 
\label{subsec aval}

La structure d'anneau (commutatif) est la structure algébrique usuelle 
d'anneau commutatif, basée sur $(1,0,+,-,\times)$.
Nous considérons une structure d'anneau comme une structure ``où on 
calcule'' et pour laquelle on utilise le seul prédicat ``$x=0$'' ~à 
l'exclusion de tous autres prédicats plus compliqués.
L'égalité $t=t'$ est elle-même considérée simplement comme une 
autre écriture pour $t-t'=0$.

Se donner une {\sl  présentation d'anneau}, c'est donner un ensemble $G$ 
de ``générateurs'' et un ensemble~$\Rzero$ de ``relations''  qui sont 
toutes de la forme $t=0$ avec $t$ un élément de $\zg$. Dans la suite 
pour simplifier, nous considérons $\Rzero$ simplement comme une partie 
de $\zg$.

La plupart des axiomes d'anneau commutatifs sont absorbés par les 
calculs dans $\zg $ et il nous reste alors les axiomes suivants, qui sont 
les règles que nous pourrons appliquer dans nos calculs.
$$\begin{array}{rlccl} 
&\vdash\;  0= 0&\qquad&\qquad&A(1) \\ 
(x = 0,\ y=0)&\vdash\;  x+y = 0&\qquad&\qquad&A(2) \\ 
x = 0 &\vdash\;   xy= 0&\qquad&\qquad& A(3) \\
\end{array}$$

Le but est de calculer, pour la présentation $(G;\Rzero)$, tous les 
termes $t\in\zg $ pour lesquels $t=0$ peut être prouvé.
Ce type de calcul ne comporte aucun embranchement, ce qui fait que nous 
sommes dans un cadre ``non dynamique'', même si on peut penser la 
structure comme une structure dynamique. En fait, la différence, 
lorsque l'on pense la structure comme dynamique, c'est qu'on ne prouve que 
des égalités $t=0$, et rien d'autre. On ne peut pas prouver, par 
exemple $1\not= 0$, parce que le prédicat $x\not= 0$ n'a pas été 
introduit, et on n'a pas dit selon quelles règles on le manipulerait.

Un anneau dynamique n'est rien d'autre qu'une présentation $(G;\Rzero)$ 
(où $\Rzero)$ est une partie de~$\zg$)~à partir de laquelle on fait 
les calculs conformément aux 3 axiomes $A(1,2,3)$ des anneaux. Du point 
de vue des égalités $t=0$, il n'y a aucune différence avec la 
structure d'anneau usuelle (non dynamique), comme le dit la proposition 
triviale suivante.

\begin{proposition} 
\label{prop faits anneau}
Soit $(G;\Rzero)$ un anneau dynamique, et $t\in\zg$. Alors $t=0$ est 
prouvable si et seulement si $t$ est dans l'idéal $\Izero$ de $\zg $ 
engendré par $\Rzero$.
\end{proposition}

Le fait que, lorsque la présentation est finie, il existe une méthode 
algorithmique pour tester la prouvabilité des faits n'a rien d'évident.

Cependant, en l'absence de toute théorie constructive des bases de 
Gr\"obner, ou bien encore dans le cas d'une présentation non finie, la 
t\^{a}che de déterminer les faits prouvables peut être grandement 
facilitée par l'usage d'un analogue du principe local-global abstrait, 
que nous pouvons formuler constructivement dans le cadre des structures 
dynamiques.

Tout d'abord nous devons introduire la notion d'anneau local dynamique.
Un {\sl  anneau local dynamique} est simplement un anneau dynamique où 
on~a le droit d'appliquer une nouvelle règle de calcul, donnée par 
l'axiome des anneaux locaux, écrit comme suit:
$$\begin{array}{rlccl} 
x + y  = 1&\vdash\;\; (\exists u\;ux =1)\;  \lor \; (\exists v\;vy=1) 
&\qquad&\qquad&AL \\ 
\end{array}$$

Comment cet axiome doit-il être appliqué~? Chaque fois qu'on~a 
prouvé, dans une branche du calcul, une égalité $t+t'=1$, on~a la 
possibilité d'ouvrir deux sous branches, dans la première un nouveau 
paramètre $u$ est introduit (i.e. un paramètre qui ne figure ni dans 
$G$ ni parmi les paramètres précédemment introduits dans la branche) 
ainsi qu'une nouvelle relation $ut=1$, dans la seconde branche on 
introduit un nouveau paramètre $v$ et la nouvelle relation $vt'=1$.
Si $P$ est l'ensemble des paramètres introduits au dessus d'un certain 
point de notre calcul arborescent, les termes qui peuvent être 
considérés~à cet endroit sont les éléments de $\Z[G\cup P]$.

Un tel calcul arborescent, arrêté au bout d'un temps fini, s'appelle 
une {\sl  évaluation dynamique} de  l'anneau local (dynamique) 
$(G;\Rzero)$. À chaque feuille du calcul ont été prouvées des 
égalités $t=0$ avec $t\in\Z[G\cup P]$.

Quand un fait $t=0$, avec $t\in\zg$, est-il déclaré prouvé pour un 
anneau local dynamique~? C'est {\sl  lorsqu'il est prouvé~à toutes les 
feuilles d'une évaluation dynamique de l'anneau local}.

Le principe local-global abstrait admet maintenant une {\sl  
interprétation} constructive: c'est l'objet de la proposition (facile 
mais non triviale) suivante.

\begin{plcd}[recollement dynamique des égalités, première version] 
\label{plcd faits anneau local} ~

\noindent 
Pour prouver un fait $t=0$ dans un anneau, vous pouvez aussi bien faire 
comme si l'anneau était local. \\ 
De manière plus formelle:\\
Soit $(G;\Rzero)$ un anneau dynamique, et $t\in\zg$. Si le fait $t=0$ est 
prouvé dans l'anneau local dynamique  $(G;\Rzero)$ alors il est 
également prouvable dans l'anneau dynamique~$(G;\Rzero)$: ajouter 
l'axiome des anneaux locaux ne permet pas de prouver plus de faits. \\
Ou si l'on préfère:  l'évaluation dynamique d'un anneau comme 
anneau local dynamique est une procédure légitime pour prouver les 
faits $t=0$.
\end{plcd}
\begin{remark} 
\label{rem plcd1}
L'énoncé précédent doit être compris de manière 
constructive: nous vous donnons une procédure uniforme qui transforme 
toute démonstration dynamique d'un fait $t=0$ dans un anneau local dynamique  de 
présentation $(G;\Rzero)$ en une démonstration dynamique du même fait $t=0$ 
dans l'anneau dynamique ayant la même présentation  $(G;\Rzero)$.
\end{remark}

\begin{proof}
Il suffit de montrer que l'utilisation une fois de l'axiome des anneaux 
locaux ne permet pas de prouver de nouveaux faits.

Soit $\Izero$ l'idéal de $\zg$ engendré par $\Rzero$ et $s,t,p\in\zg$.
Supposons que $s+t-1\in\Izero$. Supposons également sans perte de 
généralité que $s$ et $t$ ne sont pas nuls dans $\zg$. Appliquons 
l'axiome des anneaux locaux avec $s+t=1$, et supposons qu'ensuite, nous 
sachions prouver $p=0$ dans chacune des deux branches créées.

Dans la première branche on a introduit le paramètre $u$ avec la 
relation $us-1=0$, donc si on prouve $p=0$ c'est qu'on~a une égalité 
dans $\zg[u]$:
$$p\=i_0(u)+(us-1)r_0(u)
$$
avec $i_0\=i_{0,0}+i_{0,1}u+\cdots+i_{0,n}u^n\in\Izero[u]$  et 
$r_0(u)\in\zg[u]$.
Nous utilisons le symbole $\=$ pour désigner une égalité dans $\zg$, 
c.-à-d. une identité algébrique, en vue de distinguer cette 
égalité du prédicat $=0$ dans la structure algébrique dynamique.
En multlipliant par $s^n$ et en réduisant dans $i_0(u)$ les $u^ks^k$ 
modulo $us-1$, on obtient une nouvelle égalité dans $\zg[u]$:
$$s^np\=i_1+(us-1)r_1(u)
$$
avec $i_1\in\Izero$ et $r_1(u)\in\zg[u]$. Mais comme la variable $u$ ne 
figure que dans le dernier produit, on~a~$r_1=0$, et donc 
$$s^np\=i_1\eqno(1)$$
(ceci est couramment appelé le {\sl  truc de Rabinovitch}).

De la même manière, dans la seconde branche, on obtient une 
égalité 
$$t^mp\=i_2\eqno (2)$$
avec $i_2\in\Izero$.

Il reste~à recoller ces deux égalités selon la procédure qui a 
été constamment utilisée dans les démonstrations ``local-global 
concrètes''. Précisément, on considère l'égalité $s+t\=1+i_3$ 
dans $\zg$ avec $i_3\in \Izero$. Cela donne, en élevant~à la puissance 
$m+n$, 
$$as^n+bt^m\=1+i_4\eqno (3)$$
dans $\zg$ avec $i_4\in \Izero$. En combinant $(1)$, $(2)$ et $(3)$ on 
obtient 
$p\in\Izero$.
\end{proof}

\subsection{Relectures constructives d'énoncés et de démonstrations} 
\label{subsec relec const}

Muni de cette interprétation constructive  du principe local-global 
abstrait \ref{plca ring}, pouvons-nous maintenant directement traiter la 
démonstration du théorème \ref{th matproj}~?

Une inspection détaillée de cette démonstration nous montre que ce que nous 
avons~à faire se résume en deux grandes étapes:

\ni --- fournir une démonstration du lemme de la liberté locale (page 
\pageref{lelilo}) sous forme d'une démonstration par évaluation dynamique~; la 
troisième démonstration que nous avons indiquée, la démonstration à la Azuyama, est 
justement de ce type.

\ni --- dans la démonstration du théorème \ref{th matproj} utiliser le 
principe local-global dynamique \ref{plcd faits anneau local} en lieu et 
place du principe local-global abstrait~\ref{plca ring}.  

\ss {\sl  Ainsi nous avons gagné notre pari: nous obtenons une démonstration 
entièrement constructive du théorème  \ref{th matproj}, et par 
exemple, dans le cas générique, cette démonstration construit les identités 
algébriques recherchées. En outre cette démonstration est une traduction 
``mot~à mot'' de la démonstration classique. Nous avons seulement~à remplacer 
le recollement abstrait des égalités par le recollement dynamique des 
égalités. Notez aussi que, du point de vue classique, ces deux 
théorèmes de recollement sont équivalents}.

\ms Nous traiterons la question: ``comment faire avec une démonstration moins 
élémentaire (que celle par Azuyama) du lemme de la liberté locale~?'' 
dans la section~\ref{sec iclg2}.

\medskip
Signalons aussi le fait remarquable suivant (qui court-circuite notre 
constructivisation de la démonstration classique): 

\ni {\sl  la réalisation dynamique de la démonstration du lemme de la liberté 
locale dans la théorie des anneaux locaux fournit, pour un module 
projectif de type fini $M$ sur un anneau arbitraire (lorsqu cet anneau est 
évalué dynamiquement comme un anneau local), la construction d'un 
nombre fini d'éléments $s_i$ qui engendrent~$A$ comme idéal et  tels 
que les  $M_{s_i}$ sont libres.}  

En effet, cette démonstration dynamique fournit un arbre aux feuilles duquel 
``$M$ est libre (après avoir rendu inversibles suffisamment 
d'éléments de~$A$)'' et dont chaque embranchement est obtenu en rendant 
inversible un des deux éléments $s$, $t$ pour lesquels on~a prouvé 
$s+t=1$. Une inspection détaillée de la démonstration par Azumaya nous montre 
d'ailleurs que l'arbre d'évaluation dynamique~a exactement $2^n$ 
feuilles lorsque la matrice de projection $F$ est de type $n\times n$. On 
peut donc se poser la question de savoir si la borne $2^n$, obtenue par 
deux voies assez différentes, est la borne la plus 
naturelle{\footnote{~Il est difficile de qualifier cette question de 
mathématique,~à cause du mot ``naturel'' qui, ici, semble résister 
à tout interprétation en termes de foncteurs. Mais parfois les 
questions ``non mathématiques'' sont importantes en mathématiques.}} 
(bien que peut-être pas optimale) pour l'explicitation du 
théorème~\ref{th ptf-loc}.

\ms Ce n'est pas seulement le principe local-global abstrait 1 qui admet 
une interprétation constructive.

Chaque fois qu'on~a un théorème local-global d'algèbre commutative, 
c.-à-d. un énoncé du genre ``telle propriété est vraie pour 
l'anneau~$A$ et le~$A$-module $M$ si et seulement si elle est vraie en 
tous les localisés $A_{\cp}$ et $M_{\cp}$'', on lui donnera alors 
l'interprétation constructive suivante  ``telle propriété est vraie 
pour l'anneau~$A$ et le~$A$-module $M$ si et seulement si elle est vraie 
lorsque l'on se place dans un cadre dynamique et qu'on rajoute l’axiome 
des anneaux locaux''. 

Dire qu'une propriété est vraie dans un cadre dynamique signifie qu'on 
peut construire une évaluation dynamique de la situation telle qu'à 
chaque feuille de l'arbre la propriété soit démontrée vraie. 

Du point de vue classique, les deux théorèmes (le théorème 
classique et son interprétation dynamique et constructive) sont en 
général équivalents (cela dépend cependant de la propriété en 
cause). Du point de vue constructif, seul le deuxième énoncé fait 
sens. L'important, mais c'est là une thèse qui reste~à vérifier en 
pratique, c'est que la démonstration classique de l'énoncé classique se 
réécrit ``automatiquement'' comme démonstration constructive de l'énoncé 
dynamique.

Nous donnons deux exemples de tels énoncés.

Concernant les modules projectifs de type fini, on~a le théorème 
dynamique suivant qui est la version dynamique et constructive du principe 
local-global abstrait de recollement des modules projectifs.

\begin{theorem} 
\label{th dyna ptf loc libre}  Soit  $M$ un~$A$-module de présentation 
finie. Les propriétés suivantes sont équivalentes:
\begin{itemize}
  \item Le module $M$ est projectif de type fini.
  \item Lorsqu'on évalue dynamiquement~$A$ comme anneau local, le module 
$M$ est projectif de type fini.
  \item Lorsqu'on évalue dynamiquement~$A$ comme anneau local, le module 
$M$ est libre.
 \end{itemize}
\end{theorem}

Concernant les modules projectifs de rang constant, on~a le théorème 
dynamique suivant, qui constitue notre version constructive la plus 
élaborée de la proposition \ref{prop ptfrangconstant}.

\begin{theorem} 
\label{th dyna ptf rang constant}  Soit  $M$ un~$A$-module de 
présentation finie et $k$ un entier naturel. Les propriétés 
suivantes sont équivalentes:
\begin{itemize}
  \item Le module $M$ est projectif de type fini et lorsque l'on évalue 
dynamiquement~$A$ comme corps, l'espace vectoriel $M$ est de dimension 
$k$.
  \item Lorsqu'on évalue dynamiquement~$A$ comme anneau local, le 
module $M$ est libre de rang~$k$.
  \item Le module $M$ est projectif de type fini et si $F$ est une 
matrice de projection $n\times n$ ayant pour image un module isomorphe~à 
$M$, le polynôme caractéristique de $F$ est égal~à $X^{n-k}(X-1)^k$,~à des nilpotents près.
  \item Le module $M$ est projectif de type fini et si $F$ est une 
matrice de projection $n\times n$ ayant pour image un module isomorphe~à 
$M$, le polynôme caractéristique de $F$ est égal~à $X^{n-k}(X-1)^k$ 
et tous les mineurs d'ordre $k+1$ de $F$ sont nuls.
\end{itemize}
\end{theorem}

\section{Compléments sur l'inter\-préta\-tion cons\-truc\-tive du 
principe local-global} 
\label{sec iclg2}
 
Nous reprenons dans cette section la question de la relecture constructive 
de la démonstration du théorème \ref{th matproj}. Comme nous l'avons déjà 
signalé, une inspection détaillée de cette démonstration nous montre que ce 
que nous avons~à faire se résume en deux grandes étapes:

\ni --- fournir une démonstration du lemme de la liberté locale sous forme 
d'une démonstration par évaluation dynamique. 

\ni --- dans la démonstration du théorème \ref{th matproj} utiliser le 
principe local-global dynamique \ref{plcd faits anneau local} en lieu et 
place du principe local-global abstrait \ref{plca ring}.  

La troisième démonstration du lemme de la liberté locale remplit la 
première condition. Cependant, si on considère la première ou la 
deuxième démonstration du lemme de la liberté locale, on constate qu'elle 
n'est pas directement une démonstration par évaluation dynamique dans la 
théorie des anneaux locaux (telle que nous l'avons définie~à la 
section \ref{subsec aval}). Il s'agit néanmoins dans les deux cas d'une 
démonstration {\sl  élémentaire}, i.e. qui peut être développée en 
tant que démonstration formelle~à l'intérieur de la théorie du premier 
ordre des anneaux locaux. 

Le théorème 1.1 de \cite{clr}, qui est un théorème de logique (une 
variante du {\sl  théorème d'élimination des coupures}), nous permet 
de transformer toute démonstration d'un fait $t=0$ dans la théorie formelle du 
premier ordre des anneaux locaux en une démonstration par simple évaluation 
dynamique. Ainsi, nous  avons mis~à jour un contenu algorithmique 
caché pour la démonstration classique abstraite que nous avons donnée du 
théorème \ref{th matproj}, même si nous prenons la première ou la 
deuxième démonstration du lemme de la liberté locale (qui ne sont pas 
entièrement constructives). 

Pour ne pas faire appel~à ce théorème de logique, nous donnons dans 
la section \ref{subsec AInv}, le moyen de récupérer directement la 
démonstration du lemme de la liberté locale comme démonstration par évaluation 
dynamique lorsque nous utilisons la deuxième démonstration. Pour cela il nous 
faut introduire en tant que tels les  prédicats d'inversibilité et de 
non inversibilité qui figurent explicitement dans les deux premières 
démonstrations du lemme de la liberté locale.   
\subsection{Anneau avec idéal et préinversibles: définition des 
structures} 
\label{subsec AInv}

Notre premier travail consiste ici~à décrire un anneau muni d'un 
monoïde et d'un idéal, comme première approche d'un anneau local 
avec ses inversibles et son idéal maximal. 

La structure d'anneau (commutatif) avec idéal et préinversibles est la 
structure d'anneau commutatif, basée sur $(1,0,+,-,\times)$, où on 
rajoute deux prédicats: $\Ua(x)$ pour dire ``$x$ est préinversible'' 
(i.e.~$x$ est inversible modulo l'idéal), et $\Ja(x)$ pour dire ``$x$ 
est dans l'idéal’’ (c.-à-d. résiduellement nul). Nous avons déjà 
les trois axiomes $A(1),\;A(2),\;A(3)$ et nous rajoutons le système 
d'axiomes suivant.
$$\begin{array}{rlccl} 
  x=0&\vdash \;  \Ja(x)   &\qquad&\qquad& AI(1) \\
(\Ja(x),\;\Ja(y)) \; &\vdash  \;\Ja(x+y)    &\qquad&\qquad& AI(2) \\
\Ja(x)  &\vdash  \;\Ja(xy)    &\qquad&\qquad& AI(3) \\
\Ja(x^2)  &\vdash  \;\Ja(x)    &\qquad&\qquad& AI(4) \\
  &\vdash  \;  \Ua(1)&\qquad&\qquad& AU(1) \\
(\Ua(x),\;\Ja(y))&\vdash  \;  \Ua(x+y)&\qquad&\qquad& AU(2) \\
(\Ua(x),\;\Ua(y))&\vdash  \;  \Ua(xy)&\qquad&\qquad& AU(3) \\
\Ua(xy)&\vdash  \;  \Ua(x)&\qquad&\qquad& AU(4) \\
(\Ua(x),\; xy=0)&\vdash  \;  y=0&\qquad&\qquad& AU(5) \\
(\Ua(x),\; \Ja(xy)&\vdash  \;  \Ja(y)&\qquad&\qquad& AU(6) \\
\end{array}$$
Nous considérons une structure d'anneau avec idéal et préinversibles 
comme une structure dynamique, une structure ``où on calcule'' et pour 
laquelle on n'utilise que les trois prédicats ``$x=0$'', $\Ja(x)$ et  
$\Ua(x)$ {\sl ~à l'exclusion de tous autres prédicats plus 
compliqués}.
Rappelons que l'égalité $t=t'$ est elle-même considérée 
simplement comme une autre écriture pour $t-t'=0$. 

Se donner une {\sl  présentation d'anneau avec idéal et 
préinversibles}, c'est donner un ensemble $G$ de ``générateurs'' et 
un ensemble $R$ de ``relations''  qui sont toutes de la forme $t=0$ ou de 
la forme $\Ua(t)$ ou de la forme $\Ja(t)$ avec $t$ un élément de 
$\zg$. Le but du calcul est de construire des termes $t'$ tels que $t'=0$ 
ou tels que $\Ja(t')$ ou tels que $\Ua(t')$. 
Pour simplifier, nous considèrerons que la présentation est donnée 
par $G$ et par trois parties de $\zg$, $\Rzero$, $\Rua$ et $\Rja$, qui 
correspondent aux trois types de relations données dans la 
présentation.

Notez que les trois faits suivants sont équivalents: $1=0$, $\Ja(1)$ et 
$\Ua(0)$. Dans ce cas, pour tout terme $t$ les faits $t=0$, $\Ua(t)$ et 
$\Ja(t)$ sont prouvables.

Récapitluons: nous définissons la structure d'{\sl  anneau avec 
idéal et préinversibles}, comme une structure basée sur 
$(1,0,+,-,\times,\Ja,\Ua)$, et soumise aux axiomes $A(1),\ldots,A(3)$, 
$AI(1),\ldots,AI(4)$, $AU(1),\ldots,AU(6)$. Un anneau dynamique avec 
idéal et préinversibles est donné par une présentation 
$(G;\Rzero,\Rja,\Rua)$ où $\Rzero$, $\Rja$ et $\Rua$ sont trois parties 
de $\zg$.

Le lecteur pourra protester et dire que nous n'avons pas mis exactement 
les axiomes correspondant~à la structure. Nous demandons en effet que 
l'idéal soit radical, et par ailleurs nous ne donnons aucun axiome pour 
garantir que les préinversibles peuvent être inversés modulo 
l'idéal. Disons que ce n'était pas là notre but. Nous décrivons en 
fait  une bonne structure intermédiaire pour arriver~à la structure 
d'anneau local avec son idéal maximal et ses inversibles. 

En fait notre structure ``pauvre'' est intéressante parce qu'elle 
contient suffisamment d'axiomes sans toutefois comporter aucun axiome avec 
$\exists$ ni aucun axiome avec $\lor$. Cela permet de démontrer 
facilement quels sont les faits prouvables pour une structure dynamique 
donnée.

Nous introduisons maintenant la structure dynamique d'{\sl  anneau local 
avec idéal maximal et inversibles}.
C'est la structure d'anneau avec idéal et préinversibles qu'on 
évalue dynamiquement en considérant les trois axiomes 
supplémentaires suivants:

$$\begin{array}{rlccl} 
\Ua(x)    &\vdash \; \exists u\; ux=1 &\qquad&\qquad& ALMI(1) \\ 
&\vdash\;  (\Ua(x)\lor\Ja(x))&\qquad&\qquad& ALMI(2) \\
x+y=1 &\vdash \;(\Ua(x)\lor\Ua(y))&\qquad&\qquad& ALMI(3) \\ 
\end{array}$$

\begin{remark} 
\label{rem axiomes ALMI}
En fait le dernier axiome résulte facilement des précédents: 
si on~a $\Ja(x)$, puisqu’on~a $\Ua(1)$ on en déduit $\Ua(1-x)$. 
\end{remark}

La lectrice n'aura pas de mal~à se convaincre que les axiomes de la 
structure d'anneau local avec idéal maximal et inversibles sont 
exactement ceux qui caractérisent les anneaux locaux~à corps 
résiduel discret avec des prédicats spécifiant les éléments de 
l'idéal maximal et les inversibles.

En fait, on aurait pu, de manière plus naturelle, introduire la 
structure d'anneau local avec idéal maximal et inversibles en donnant 
seulement cinq axiomes qui traduisent la définition des éléments 
inversibles, des éléments non inversibles et l'axiome des anneaux 
locaux. Sans introduire les axiomes $AI$ ni les axiomes $AU$ on aurait 
simplement pris les trois axiomes $ALMI$ ci-dessus et les deux suivants:
$$\begin{array}{rlccl} 
xy=1    &\vdash \; \Ua(x) &\qquad&\qquad& ALMI(1bis) \\ 
(\Ua(x),\Ja(x))&\vdash\;  1=0 &\qquad&\qquad& ALMI(2bis) \\
\end{array}$$

\subsection{Faits prouvables et interprétation du principe local-global 
abstrait} 
\label{subsec Faits prouvables}

\begin{proposition} 
\label{prop faits AIU}
Soit $(G;\Rzero,\Rja,\Rua)$ un anneau avec idéal et préinversibles, 
dynamique, et soit $t\in\zg$. 
Soit $\Izero$ l'idéal de $\zg$ engendré par $\Rzero$, $\Ija$ l'idéal 
de $\zg$ engendré par $\Rja$ et~$\Mua$ le monoïde multiplicatif de~$\zg$ engendré par $\Rua$.
 
\ni Alors:
\begin{itemize}
\item  $t=0$ est prouvable si et seulement si on~a dans $\zg$ une 
égalité du type
$$(u+j)t+i\=0
$$
avec $u\in\Mua$, $j\in\Ija$, et $i\in\Izero$. 

\item  $\Ja(t)$ est prouvable si et seulement si on~a dans $\zg$ une 
égalité du type
$$ut^n+j+i\=0
$$
avec $n\in\N$, $u\in\Mua$, $j\in\Ija$ et $i\in\Izero$. 

\item  $\Ua(t)$ est prouvable si et seulement si on~a dans $\zg$ une 
égalité du type
$$u+j+at+i\=0
$$
avec $u\in\Mua$, $j\in\Ija$, $a\in\zg$ et $i\in\Izero$. 

\end{itemize}
\end{proposition}

Un corollaire immédiat est le suivant (dans la lignée du théorème 
1.1 de \cite{clr}: on peut toujours rajouter des nouveaux prédicats~à 
condition de les soumettre~à des axiomes ``logiques'' raisonnables).
\begin{corollary} 
\label{cor faits anneau AIU} 
Si un anneau dynamique $(G;\Rzero)$ est vu comme un anneau avec idéal et 
préinversibles dynamique $(G;\Rzero,\emptyset,\emptyset)$, les faits 
prouvables $t=0$ sont les mêmes pour les deux structures dynamiques.
\end{corollary}

Un autre corollaire remarquable et immédiat est le suivant.
\begin{corollary} 
\label{cor faits et collapsus AIU}
Dans un anneau dynamique avec idéal et préinversibles: 

\ni --- a) un fait $\Ua(t)$ est prouvable si et seulement si  la relation 
$\Ja(t)$ (rajoutée dans la présentation) rend prouvable $1=0$,  

\ni --- b) un fait $\Ja(t)$ est prouvable si et seulement si  la relation 
$\Ua(t)$ (rajoutée dans la présentation) rend prouvable $1=0$.  
\end{corollary}

\begin{proof}[Démonstration de la proposition]
On voit facilement que les conditions sont suffisantes. 
\\
Pour voir qu'elles sont nécessaires, il suffit de vérifier que les 
éléments de $\Rzero$, $\Rja$ et $\Rua$ sont ``conformes'' et que chaque 
axiome produit des éléments ``conformes''~à partir d'éléments 
``conformes''.
La plupart des calculs ne présentent aucune difficulté. Nous traitons 
les cas des axiomes $AI(2)$, $AU(2)$ et~$AU(6)$.

\sni
{\sl  Cas de l'axiome $AI(2)$}.
On suppose que l'on~a deux faits prouvables ``conformes'' $\Ja(t_1)$ et 
$\Ja(t_2)$, c.-à-d. qu'on~a deux égalités dans $\zg$
$$\begin{array}{c} 
 u_1t_1^m+j_1+i_1\=0  \\ 
 u_2t_2^n+j_2+i_2\=0  \\ 
\end{array}$$
(avec les mêmes conventions que dans l'énoncé pour $u$, $j$ et $i$), 
on en déduit
$$\begin{array}{c} 
u_1u_2(t_1+t_2)^{m+n}\;\=\;u_1u_2(at_1^m+bt_2^n)\;\=\;
(u_2a)(u_1t_1^m)+(u_1b)(u_2t_2^n)\;\=   \\ 
u_2a(-j_1-i_1)+u_1b(-j_2-i_2)\;\=\;-j_3-i_3   \\
\end{array}$$
Et le fait prouvable $\Ja(t_1+t_2)$ est donc bien ``conforme''.

\smallskip\ni
{\sl  Cas de l'axiome $AU(2)$}.
On suppose qu'on~a deux faits prouvables ``conformes'' $\Ja(t_1)$ et 
$\Ua(t_2)$, c.-à-d. deux égalités dans $\zg$
$$\begin{array}{c} 
 u_1t_1^m+j_1+i_1\=0  \\ 
 u_2+j_2+a_2t_2+i_2\=0  \\ 
\end{array}$$
d'où
$$u_1u_2+a_2u_1(t_1+t_2)+j_3+i_3\;\=\;a_2u_1t_1 
$$
on élève~à la puissance $m$, dans chaque membre on regroupe 
judicieusement les termes, on obtient
$$u_4+a_4(t_1+t_2)+j_4+i_4\;\=\;a_5u_1t_1^m\;\=\;-a_5(j_1+i_1);\=\;j_5+i_5
$$
Et le fait prouvable $\Ua(t_1+t_2)$ est donc bien ``conforme''.

\smallskip\ni
{\sl  Cas de l'axiome $AU(6)$}.
On suppose qu'on a deux faits prouvables ``conformes'' $\Ua(t_1)$ et 
$\Ja(t_1t_2)$, c.-à-d. deux égalités dans $\zg$
$$\begin{array}{c} 
 u_1+j_1+a_1t_1+i_1\=0  \\ 
 u_2(t_1t_2)^m+j_2+i_2\=0  \\ 
\end{array}$$
Dans la première égalité, on fait passer $a_1t_1$ dans le second 
membre, on élève~à la puissance $m$ et on regroupe judiceusement les 
termes, on obtient
$$u_3+j_3+i_3\;\=\;a_3t_1^m
$$
puis on multiplie par $u_2t_2^m$, cela donne
$$u_4t_2^m+j_4+i_4\;\=\;a_3u_2(t_1t_2)^m\;\=\;j_5+i_5
$$
Et le fait prouvable $\Ja(t_2)$ est donc bien ``conforme''.
\end{proof}

\begin{proposition} 
\label{prop faits AIU ALMI}
Soit $A=(G;\Rzero,\Rja,\Rua)$ un anneau avec idéal et préinversibles 
dynamique et $t\in\zg$. \\
Si on l'évalue dynamiquement comme anneau local avec idéal maximal et 
inversibles, tout fait prouvé (du type $t=0$ ou $\Ja(t)$ ou $\Ua(t)$) 
peut également être prouvé sans recours aux trois axiomes 
supplémentaires $ALMI(1,2,3)$. 
\end{proposition}

\begin{proof}
Vues{\footnote{~Accord de genre avec le plus proche cité.}} 
 la remarque \ref{rem axiomes ALMI} et 
le corollaire \ref{cor faits et collapsus AIU}, 
il suffit de montrer que l'utilisation une fois de l'axiome $ALMI(1)$ ou 
de l'axiome $ALMI(2)$ ne change pas les faitrs prouvés de la forme 
$p=0$. Pour $ALMI(1)$ c'est l'usuel truc de Rabinovitch. 
\\
Voyons $ALMI(2)$. On considère un terme $t$ et on ouvre deux branches, 
l'une avec $\Ua(t)$ et l'autre avec $\Ja(t)$. On part de deux égalités 
dans $\zg$ qui correspondent~à la prouvabilité de $p=0$  
respectivement dans chacune des deux branches:
\begin{eqnarray} 
(u_1t^m+j_1)p+i_1\;\=\;0       \\[.3em] 
(u_2+j_2-ta_2)p+i_2\;\=\;0       
\end{eqnarray}
avec $u_h\in\Mua$, $j_h\in\Ija$, $a_h\in\zg$ et $i_h\in\Izero$.
\\
On se base sur l'identité $(u-s)\times$(quelque chose)$\;\=\;(u^m-s^m)$ 
avec $u:=u_2+j_2$ et $s:=ta_2$. En multipliant $(2)$ par ``quelque chose'', 
on obtient
$$(u_3+j_3-t^ma_3)p+i_3\;\=\;0\eqno (3)
$$

\ni On multiplie $(1)$ par $a_3$, on obtient:
$$(a_3u_1t^m+j_4)p+i_4\;\=\;0\eqno (4)
$$

\ni On multiplie $(3)$ par $u_1$, on obtient: 
$$(u_5+j_5-a_3u_1t^m)p+i_5\;\=\;0\eqno (5)
$$

\ni Enfin on additionne $(4)$ et $(5)$, on obtient:
$$(u_5+j_6)p+i_6\;\=\;0   
$$

\ni ce qui est l'égalité cherchée.
\end{proof}

Nous voici en état de prouver constructivement une nouvelle forme 
concrète, un peu plus sophistiquée que le principe local-global 
dynamique \ref{plcd faits anneau local}, du principe local-global abstrait 
\ref{plca ring}. Cette fois-ci, les inversibles sont pris en compte.

\begin{plcd}[recollements dynamiques, deuxième  version]  
\label{plcd faits AIU ALMI}~

\noindent 
Soit $A=(G;\Rzero,\emptyset,\emptyset)$ un anneau avec idéal et 
préinversibles dynamique (avec $\Rja$ et $\Rua$ vides) et soit $t\in\zg$. 
\begin{itemize}
\item {\em  (Recollement dynamique des égalités, deuxième 
version)} Pour prouver un fait $t=0$ dans l’anneau $A$, vous pouvez aussi 
bien faire comme si l'anneau était local, en utilisant les prédicats 
d'inversibilité et non inversibilité. \\ 
De manière plus formelle: un fait du type $t=0$ est prouvable dans la 
structure d'anneau local dynamique avec idéal maximal et inversibles si 
et seulement si il est prouvable dans~$A$ comme anneau dynamique.
\item {\em  (Recollement dynamique des inversibles)} 
Un fait du type $\Ua(t)$ est prouvable dans la structure d'anneau local 
dynamique avec idéal maximal et inversibles si et seulement si $t$ est 
inversible dans~$A$ comme anneau dynamique, i.e., s'il existe un $u\in\zg$ 
avec $tu=1$ prouvable.\\ 
De manière moins formelle: Pour prouver un fait ``$t$ est inversible'' 
dans un anneau, vous pouvez aussi bien faire comme si l'anneau était 
local, en utilisant les prédicats d'inversibilité et de non 
inversibilité.
\item {\em  (Une caractérisation dynamique des nilpotents)}\\ 
Un fait du type $\Ja(t)$ est prouvable dans la structure d'anneau local 
dynamique avec idéal maximal et inversibles si et seulement si $t$ est 
nilpotent dans~$A$ comme anneau dynamique, i.e., s'il existe un entier 
naturel $m$ avec $t^m=0$ prouvable.
\end{itemize}
\end{plcd}

\begin{proof}
Cela résulte de la proposition \ref{prop faits AIU ALMI}  et de la 
caractérisation donnée dans la proposition~\ref{prop faits AIU}.
 \end{proof}
\begin{remark} 
\label{rem plcd faits AIU ALMI}
Comme conséquence du recollement dynamique des égalités dans 
le théorème précédent, si  $(G;\Rzero)$  est un anneau local 
dynamique, tout fait prouvé en le considérant comme un anneau local 
avec idéal maximal et inversibles peut être prouvé dans la 
structure d'anneau local (on peut toujours rajouter des nouveaux 
prédicats~à condition de les soumettre~à des axiomes ``logiques'' 
raisonnables). 
\end{remark}

\subsection{Récapitulons} 
\label{subsec Recapitulons}

On récapitule sur la relecture constructive  dynamique de la démonstration du théorème~\ref{th matproj}.  Un aspect un peu déroutant est que, une 
fois qu'on dispose des prédicats d'inversibilité et non 
inversibilité pour un anneau local, 
c.-à-d. en fait des prédicats d'égalité~à zéro et de non 
égalité~à zéro dans le corps résiduel, la théorie de la 
dimension des espaces vectoriels sur les corps, vue comme théorie du 
rang des matrices, et nécessaire pour la {\sl  première démonstration} du 
lemme de la liberté locale{\footnote{~Nous pensons par exemple~à la 
partie soulignée de la  phrase suivante: On considère alors un mineur 
résiduellement non nul d'ordre maximum $k$ dans $F$, et de même  un 
mineur résiduellement non nul \underline{d'ordre maximum $n-k$} dans 
$\I_n-F$.}}, n'est pas si simple~à formuler et~à prouver dynamiquement 
sans recours aux prédicats de dépendance linéaire et 
d'indépendance linéaire. Il nous faut introduire des disjonctions de 
conjonctions: un mineur d'ordre $k$ non nul et tous les mineurs d'ordre 
$k+1$ nuls\dots Cela demande donc un travail qui est faisable, mais qu'on ne 
prendra pas la peine de faire ici. \\
Par contre, toujours pour le lemme de la liberté locale, la {\sl  
deuxième démonstration} que nous avons donnée se lit très aisément comme 
une démonstration par évaluation dynamique pour la structure d’anneau local avec inversibles et idéal maximal. Ceci, joint~à la démonstration du principe 
local-global dynamique \ref{plcd faits AIU ALMI} fournit une démonstration 
constructive du théorème  \ref{th matproj}, et par exemple construit 
les identités algébriques recherchées dans le cas générique.

\addcontentsline{toc}{section}{Références}



\end{document}